\documentclass[10pt]{article}
\usepackage[margin=1.5in]{geometry}  
\usepackage{amsmath}
\usepackage{cases}
\usepackage{stmaryrd}
\usepackage{latexsym,amsfonts,amssymb,amsmath,amsthm}
\usepackage{mathrsfs}
\usepackage{graphicx}
\usepackage{verbatim}

\newtheorem{thm}{Theorem}[section]

\newtheorem{lem}[thm]{Lemma}

\newtheorem{rem}[thm]{Remark}

\numberwithin{equation}{section}

\def\C{\mathbb{C}}
\def\E{\mathbb{E}}
\def\N{\mathbb{N}}
\def\P{\mathbb{P}}
\def\R{\mathbb{R}}

\def\Z{\mathbb{Z}}

\def\CC{\mathcal{C}}
\def\DD{\mathcal{D}}
\def\HH{\mathcal{H}}
\def\SS{\mathcal{S}}

\def\VV{\mathcal{V}}
\def\NN{\mathcal{N}}
\def\QQ{\mathcal{Q}}

\def\supp{\text{\rm supp}}

\def\per{\rm per}

\def\lan{\langle}
\def\ran{\rangle}
\def\da{\downarrow}
\def\ra{\rightarrow}
\def\bs{\backslash}

\def\pa{\partial}
\def\ol{\overline}

\def\al{\alpha}

\def\ep{\epsilon}
\def\ka{\kappa}

\def\la{\lambda}
\def\La{\Lambda}
\def\si{\sigma}
\def\Si{\Sigma}
\def\om{\omega}
\def\Om{\Omega}
\def\de{\delta}
\def\De{\Delta}

\def\Ga{\Gamma}
\def\vp{\varphi}

\begin{document}

\nocite{*}

\title{Lifshitz Tails for Anderson Models with Sign-Indefinite Single-Site Potentials}

\author{Zhongwei Shen\footnote{Email: zzs0004@auburn.edu}\\Department of Mathematics and Statistics\\Auburn University\\Auburn, AL 36849\\USA}

\date{}

\maketitle

\begin{abstract}
We study the spectral minimum and Lifshitz tails for continuum random Schr\"{o}dinger operators of the form  
\begin{equation*}
H_{\om}=-\De+V_{0}+\sum_{i\in\Z^{d}}\om_{i}u(\cdot-i),
\end{equation*}
where $V_{0}$ is the periodic potential, $\{\om_{i}\}_{i\in\Z^{d}}$ are i.i.d random variables and $u$ is the sign-indefinite impurity potential. Recently, this model has been  proven to exhibit Lifshitz tails near the bottom of the spectrum under the small support assuption of $u$ and the reflection symmetry assumption of $V_{0}$ and $u$.  We here drop the reflection symmetry assumption of $V_{0}$ and $u$. We first give characterizations of the bottom of the spectrum. Then, we show the existence of Lifshitz tails in the regime where the characterization of the bottom of the spectrum is explicit. In particular, this regime covers the reflection symmetry case.
\\
Keywords: spectral minimum, Lifshitz tail, continuum Anderson model.\\
2010 Mathematics Subject Classification: 35P20, 46N50, 47B80.
\end{abstract}

\tableofcontents

\section{Introduction}\label{sec-intro}

This paper is concerned with the spectral minimum and Lifshitz tails for the following random operator
\begin{equation}\label{main-model}
H_{\om}=-\De+V_{0}+V_{\om}\quad\text{on}\quad\R^{d},
\end{equation}
where $d\in\N$, $V_{0}$ is the background potential and $V_{\om}$ is the random potential of alloy type, that is, $V_{\om}$ has the form $V_{\om}=\sum_{i\in\Z^{d}}\om_{i}u(\cdot-i)$. We assume
\begin{itemize}
\item[\rm(H1)] $V_{0}\in L^{p}_{\rm loc}(\R^{d})$ is $\Z^{d}$-periodic with $p>d$.

\item[\rm(H2)] The single-site potential $u:\R^{d}\rightarrow\R$ is in $L^{p}(\R^{d})$ with $p>d$ and supported in $\CC_{0}=\big(-\frac{1}{2},\frac{1}{2}\big)^{d}$. Both the positive part $u_{+}$ and the negative part $u_{-}$ are non-trivial.

\item[\rm(H3)] $\{\om_{i}\}_{i\in\Z^{d}}$ are independent and identically distributed (i.i.d) random variables on some probability space $(\Om,\mathcal{B},\P)$ with common distribution $\P_{0}$. The support of $\P_{0}$, denoted by $\supp(\P_{0})$, is compact and contains at least two points. 
\end{itemize}

By the canonical realization of stochastic processes, we take $\Om=(\supp\P_{0})^{\Z^{d}}$, and thus, $\P$ is the product measure $\otimes_{i\in\Z^{d}}\P_{0}$. We denote by $\E$ the expectation corresponding to $\P$.

Under $\rm(H1)$, $\rm(H2)$ and $\rm(H3)$, $H_{\om}$ is almost surely self-adjoint on $H^{2}(\R^{d})$ and $\Z^{d}$-ergodic, and hence, $\si(H_{\om})=\Si$ a.e. $\om\in\Om$ for some $\Si\subset\R$ (see e.g. \cite{CL90,Ki89}). Let
\begin{equation*}
E_{0}=\inf\Si.
\end{equation*}
Set $a=\inf\supp(\P_{0})$ and $b=\sup\supp(\P_{0})$. Then $a<b$ by $\rm(H3)$.

It is well-known (see e.g. \cite{Ki89,KM07,KW05,St01}) that under quite general assumptions on $u$ (but $u$ is sign-definite) the integrated density of states (IDS) of $H_{\om}$ exists Lifshitz tails near $E_{0}$, which is given by $\inf\si(H_{a})$ if $u\geq0$ and by $\inf\si(H_{b})$ if $u\leq0$, where
\begin{equation}\label{periodic-operators}
H_{t}=-\De+V_{0}+t\sum_{i\in\Z^{d}}u(\cdot-i),\quad t\in[a,b].
\end{equation}
The characterization of $E_{0}$, which is given as above in the case of $u$ being sign-definite, becomes a problem when $u$ changes its sign, since $H_{\om}$ no longer depends monotonously on $\om=\{\om_{i}\}_{i\in\Z^{d}}$. Moreover, due to this non-monotonous dependence, the existence or non-existence of Lifshitz tails for $H_{\om}$ is unknown for a quite long period until the recent work \cite{KN09,KN10} of Klopp and Nakamura.

To motivate the current paper, we roughly describe the results obtained in \cite{KN09} by Klopp and Nakamura. Under assumptions $\rm(H1)$, $\rm(H2)$ and $\rm(H3)$, and an additional reflection symmetry assumption on $V_{0}$ and $u$, that is, 
\begin{equation}\label{reflection-symmetric}
\begin{split}
V_{0}(x)&=V_{0}((-1)^{\tau_{1}}x_{1},\dots(-1)^{\tau_{d}}x_{d}),\\
u(x)&=u((-1)^{\tau_{1}}x_{1},\dots(-1)^{\tau_{d}}x_{d})
\end{split}
\end{equation}
for any $(\tau_{1},\dots,\tau_{d})\in\{0,1\}^{d}$ and any $x=(x_{1},\dots,x_{d})\in\R^{d}$, they proved a characterization of the bottom of the spectrum. More precisely, denote by $H^{N}_{t}$ the restriction of $H_{t}$ to $L^{2}(\CC_{0})$ with Neumann boundary condition on $\partial\CC_{0}$ and by $E(t)$ the ground state energy of $H_{t}^{N}$. They proved $E_{0}=\min\big\{E(a),E(b)\big\}$. Then, using this characterization of $E_{0}$ and an operator theoretical trick (a comparison method), they showed that the IDS of \eqref{main-model} exhibits Lifshitz tails near $E_{0}$ if $E(a)\neq E(b)$. They also constructed an interesting Bernoulli model showing that Lifshitz tails may fail when $E(a)=E(b)$. Later, they proved in \cite{KN10}, using a quite different method, the existence of Lifshitz tails near $E_{0}$ in the case $E(a)=E(b)$ with additional weak assumptions.

Inspired by the work of Klopp and Nakamura \cite{KN09,KN10}, we study the spectral minimun and Lifshitz tails for the model \eqref{main-model} without the reflection symmetry assumption \eqref{reflection-symmetric} on $V_{0}$ and $u$. After dropping this reflection symmetry assumption, Neumann operators working very well in the reflection symmetry case do not work anymore. A natural substitute for Neumann boundary condition is the so-called Mezincescu boundary condition (see e.g. \cite{Me87}). To be more specific, let $\vp\in C^{1}(\R^{d})$ be real-valued, strictly positive and $\Z^{d}$-periodic. Let $\textbf{n}_{0}:\pa\CC_{0}\ra\R^{d}$ be the outer normal of $\pa\CC_{0}$ and define $\chi_{0}:\pa\CC_{0}\ra\R$ by setting
\begin{equation*}
\chi_{0}(x)=-\frac{1}{\vp(x)}\textbf{n}_{0}(x)\cdot\nabla\vp(x),\quad x\in\partial\CC_{0}.
\end{equation*}
For $t\in[a,b]$, denote by $H_{t,\CC_{0}}$ the restriction of $H_{t}$ (given in \eqref{periodic-operators}) to $L^{2}(\CC_{0})$ with Mezincescu boundary condition defined via $\chi_{0}$ (or $\vp$) (see Section \ref{Mezin-bd} for the definition) on $\partial\CC_{0}$ and by $E_{\vp}(t)$ the ground state energy of $H_{t,\CC_{0}}$. Note we should use the notation $H_{t,\CC_{0}}^{\chi_{0}}$ (or may be more precisely $H_{t,\CC_{0}}^{\vp}$, since $\CC_{0}$ and $\vp$ determine $\chi_{0}$), but we here suppress the superscript, and we will use suppressed notations in the sequel.

Here, we are particularly interested in the cases $\vp=\vp_{a}$ and $\vp=\vp_{b}$, where $\vp_{a}$ and $\vp_{b}$ are ground states of $H_{a}$ and $H_{b}$, respectively. We point out that $\vp_{a}$ and $\vp_{b}$ can be chosen to be continuously differentiable and strictly positive under the assumption that $p>d$ (see e.g. \cite[Theorem C.2.4]{Si82}). In fact, we can relax this assumption and require only $p>\frac{d}{2}$ in $\CC_{0}$ (see \cite[Remarks 2.9 \rm(iii)]{KW05}). Therefore, if $V_{0}\equiv0$, we only need to assume $p>\frac{d}{2}$ since $u$ is supported in $\CC_{0}$.

As in \cite{KN09}, we can use $E_{\vp}(t)$ to characterize the bottom of the spectrum $E_{0}$. This gives our first main result.

\begin{thm}\label{thm-spectral-minimum}
Suppose $\rm (H1)$, $\rm(H2)$ and $\rm(H3)$. 
\begin{itemize}
\item[$\rm(i)$] If $E_{\vp_{a}}(a)\leq E_{\vp_{a}}(b)$, then $E_{0}=E_{\vp_{a}}(a)$.
\item[$\rm(ii)$] If $E_{\vp_{b}}(a)\geq E_{\vp_{b}}(b)$, then $E_{0}=E_{\vp_{b}}(b)$.
\item[$\rm(iii)$] If $E_{\vp_{a}}(a)>E_{\vp_{a}}(b)$ and $E_{\vp_{b}}(a)<E_{\vp_{b}}(b)$, then
\begin{equation*}
E_{0}\in\big[E_{\vp_{a}}(b),E_{\vp_{a}}(a)\big]\cap\big[E_{\vp_{b}}(a),E_{\vp_{b}}(b)\big]
\end{equation*}
In particular, there exist $t_{a},t_{b}\in[a,b]$ such that $E_{\vp_{a}}(t_{a})=E_{0}=E_{\vp_{b}}(t_{b})$.
\end{itemize}
\end{thm}

The proof of the above theorem is given in Subsection \ref{sec-spectral-minimum}. We next study the existence of Lifshitz tails. The lower bound with exponent $\frac{d}{2}$ has been established in \cite[Theorem 0.2]{KN09}. We here focus on the upper bound in the case of Theorem \ref{thm-spectral-minimum}$\rm(i)$ and $\rm(ii)$. Due to technical reasons, we consider the following three cases:
\begin{itemize}
\item[\rm(I)] $E_{\vp_{a}}(a)<E_{\vp_{a}}(b)$ or $E_{\vp_{b}}(a)>E_{\vp_{b}}(b)$;
\end{itemize}
\begin{itemize}
\item[\rm(II)] $E_{\vp_{a}}(a)=E_{\vp_{a}}(b)$ and $E_{\vp_{b}}(a)=E_{\vp_{b}}(b)$;
\end{itemize}
\begin{itemize}
\item[\rm(III)] $E_{\vp_{a}}(a)=E_{\vp_{a}}(b)$ and $E_{\vp_{b}}(a)<E_{\vp_{b}}(b)$, or, $E_{\vp_{a}}(a)>E_{\vp_{a}}(b)$ and $E_{\vp_{b}}(a)=E_{\vp_{b}}(b)$;
\end{itemize}

If $V_{0}$ and $u$ are reflection symmetric as considered in \cite{KN09,KN10}, then $\nabla\vp_{a}$ and $\nabla\vp_{b}$ vanish on $\partial\CC_{0}$. This, in particular, says that Mezincescu boundary conditions defined via $\vp_{a}$ and $\vp_{b}$ reduce to Neumann boundary condition, and hence, $\rm(I)$ and $\rm(II)$ cover all the possibilities.

Before stating corresponding results, we need
\begin{itemize}
\item[\rm(H4)] $V_{0}$ and $u$ are bounded from below.
\end{itemize}
An interpretation of $\rm(H4)$ is as follows: we will define the IDS using eigenvalue counting functions of operators with Mezincescu boundary conditions. The lower boundedness of $V_{0}$ and $u$ then ensures that such defined IDS coincides with the one with usual definition. See \cite[Theorem 1.3]{Mi02}.

Now, for the upper bound in the case $\rm(I)$, we prove in Section \ref{sec-lifshitz-tail} the following

\begin{thm}\label{thm-upper-bound}
Suppose $\rm (H1)$, $\rm(H2)$, $\rm(H3)$ and $\rm(H4)$. If either $E_{\vp_{a}}(a)<E_{\vp_{a}}(b)$ or $E_{\vp_{b}}(a)>E_{\vp_{b}}(b)$ is satisfied, then
\begin{equation*}
\limsup_{E\da E_{0}}\frac{\ln|\ln N(E)|}{\ln(E-E_{0})}\leq-\frac{d}{2}.
\end{equation*}
\end{thm}

To state the results in cases $\rm(II)$ and $\rm(III)$, we first make a convention: the Mezincescu boundary condition is defined via $\vp_{a}$ if $E_{\vp_{a}}(a)=E_{\vp_{a}}(b)$ and $E_{\vp_{b}}(a)\leq E_{\vp_{b}}(b)$, and defined via $\vp_{b}$ if $E_{\vp_{a}}(a)\geq E_{\vp_{a}}(b)$ and $E_{\vp_{b}}(a)=E_{\vp_{b}}(b)$. For the case $\rm(II)$, we need

\begin{itemize}
\item[\rm(H5)] Set $\SS_{0}=(-\frac{1}{2},\frac{1}{2})^{d-1}\times(-\frac{1}{2},\frac{3}{2})$. Consider $H_{\om,\SS_{0}}$ with $\om_{(0,0)},\om_{(0,1)}\in\{a,b\}$ and $\om_{(0,0)}\neq\om_{(0,1)}$. We assume
\begin{equation*}
\inf\si(H_{\om,\SS_{0}})>E_{0}.
\end{equation*}
\end{itemize}

\begin{thm}\label{thm-Lifshitz-tail-non-critical}
Suppose $\rm (H1)$, $\rm(H2)$, $\rm(H3)$, $\rm(H4)$ and $\rm(H5)$. If
$E_{\vp_{a}}(a)=E_{\vp_{a}}(b)$ and $E_{\vp_{b}}(a)=E_{\vp_{b}}(b)$
are satisfied, then
\begin{equation*}
\limsup_{E\da E_{0}}\frac{\ln|\ln N(E)|}{\ln(E-E_{0})}\leq-\frac{1}{2}.
\end{equation*}
\end{thm}

For case $\rm(III)$, we also need

\begin{itemize}
\item[\rm(H6)]  Set $\SS_{0}=(-\frac{1}{2},\frac{1}{2})^{d-1}\times(-\frac{1}{2},\frac{3}{2})$. Consider $H_{\om,\SS_{0}}$ with $\om_{(0,0)}=\om_{(0,1)}=b$ if $E_{\vp_{a}}(a)=E_{\vp_{a}}(b)$ and $E_{\vp_{b}}(a)<E_{\vp_{b}}(b)$, and $\om_{(0,0)}=\om_{(0,1)}=a$ if $E_{\vp_{a}}(a)>E_{\vp_{a}}(b)$ and $E_{\vp_{b}}(a)=E_{\vp_{b}}(b)$. Denote by $\vp_{\SS_{0}}$ the ground state of $H_{\om,\SS_{0}}$. If $\inf\si(H_{\om,\SS_{0}})=E_{0}$, we assume
\begin{equation*}
\nu=1,
\end{equation*}
where the constant $\nu>0$ satisfies $\vp_{\SS_{0}}|_{\CC_{(0,1)}}=\nu\vp_{\SS_{0}}|_{\CC_{(0,0)}}(\cdot-(0,1))$.
\end{itemize}

\begin{thm}\label{thm-Lifshitz-tail-non}
Suppose $\rm (H1)$, $\rm(H2)$, $\rm(H3)$, $\rm(H4)$, $\rm(H5)$ and $\rm(H6)$. If either
\begin{equation*}
E_{\vp_{a}}(a)=E_{\vp_{a}}(b)\quad\text{and}\quad E_{\vp_{b}}(a)<E_{\vp_{b}}(b)
\end{equation*}
or
\begin{equation*}
E_{\vp_{a}}(a)>E_{\vp_{a}}(b)\quad\text{and}\quad E_{\vp_{b}}(a)=E_{\vp_{b}}(b)
\end{equation*}
is satisfied, then
\begin{equation*}
\limsup_{E\da E_{0}}\frac{\ln|\ln N(E)|}{\ln(E-E_{0})}\leq-\frac{1}{2}.
\end{equation*}
\end{thm}

If $\rm(H5)$ is true, then it is true for any domian $\SS_{0}+i$ for $i\in Z^{d}$ due to the unitary equivalence. Moreover, in $\rm(H5)$, we choose the $d$-th coordinate to make this assumption and it's easy to see this choice does not lose the generality. Also, $\rm(H5)$ is necessary for Bernoulli molels with reflection symmetric potentials to exhibit Lifshitz tails (see \cite{KN10}). In Subsection \ref{subsec-example}, we will use Klopp and Nakamura's example constructed in \cite{KN09} to explain that Lifshitz tails may fail if $\rm(H5)$ fails. For $\rm(H6)$, it is shown in Lemma \ref{lem-key} that if $\inf\si(H_{\om,\SS})=E_{0}$, then there exists $\nu>0$ such that $\vp_{\SS}|_{\CC_{(0,1)}}=\nu\vp_{\SS}|_{\CC_{(0,0)}}(\cdot-(0,1))$. 

The proofs of Theorem \ref{thm-Lifshitz-tail-non-critical} and Theorem \ref{thm-Lifshitz-tail-non} are given in Section \ref{sec-LT-non-optimal-upper-bound}. Their proofs can be treated similarly, except for Lemma \ref{lem-key-2}. This is the reason why we need $\rm(H6)$ for the case $\rm(III)$. Due to technical reasons, we will consider non-Bernoulli models and Bernoulli models separately.

Due to the sign-indefiniteness of $u$, the model \eqref{main-model} is a special non-monotonous model. Non-monotonous models, like models with random magnetic fields (see e.g \cite{Gh07}, \cite{FNNN03}, \cite{Na00-1}, \cite{Na00-2}) and random displacement model (see \cite{KN10,KLNS12-1,KLNS12-2}), have been shown to exhibit Lifshitz tails. In \cite{Gh08}, Lifshitz tails were proven to exist at open band edges for models similar to \eqref{main-model} with further assumptions on the spectrum of the background operator.

The rest of the paper is organized as follows. In Subsection \ref{Mezin-bd}, we collect some results about Schr\"{o}dinger operators restricted to subdomains with Mezincescu Boundary Conditions. In Subsection \ref{sec-spectral-minimum}, we characterize the bottom of the spectrum, that is, we prove Theorem \ref{thm-spectral-minimum}. In Subsection \ref{subsec-IDS}, we present the existence and uniquess of the IDS. Section \ref{sec-lifshitz-tail} is devoted to the proof of Theorem \ref{thm-upper-bound}. In Section \ref{sec-lower-bound-gse}, we prove lower bound estimates for ground state energies of some well constructed operators. These estimates serve as a preparation for the proofs of Theorem \ref{thm-Lifshitz-tail-non-critical} and Theorem \ref{thm-Lifshitz-tail-non}, which are given in Section \ref{sec-LT-non-optimal-upper-bound}. In Section \ref{sec-dis}, we give some further discussions. 

Throughout the paper, we use the following notations: if the spectrum of a lower bounded self-adjoint operator $H$ consists of eigenvalues, we denoted them by $E_{0}(H)\leq E_{1}(H)\leq E_{2}(H)\cdots$; $\lan\cdot,\cdot\ran$ ($\|\cdot\|$) denotes the inner product (norm) on various spaces of square integrable complex functions; $\#\{\cdot\}$ denotes the cardinal number of the set $\{\cdot\}$; and if $O$ is a subdomain in $\R^{d}$, its boundary is denoted by $\pa O$; a self-adjoint operator $H$ on $L^{2}(\R^{d})$ restricted to $L^{2}(O)$ with various Mezincescu boundary conditions are denoted by $H_{O}$; denote by $\N$ the positive natural numbers and set $\N_{0}=\N\cup\{0\}$.


\section{Spectral Minimum and IDS}

In this section, we first review some basic properties of opeators with Mezincescu boundary conditions in Subsection \ref{Mezin-bd}, which are then used to provide characterizations of the bottom of the spectrum in Subsection \ref{sec-spectral-minimum}. Opeators with Mezincescu boundary conditions also provide an alternative way to the definition of the IDS, which is given in Subsection \ref{subsec-IDS}.

\subsection{Operators with Mezincescu Boundary Conditions}\label{Mezin-bd}

We collect some results about Schr\"{o}dinger operators restricted to subdomains with Mezincescu boundary conditions. It is referred to \cite{KW05,Me87} for more discussions.

Let $\La\subset\R^{d}$ be a $d$-dimensional open cube centered at $0$ with integer side length. For $\chi_{\La}:\pa\La\ra\R$ in $L^{\infty}(\pa\La)$, define $\QQ:H^{1}(\La)\times H^{1}(\La)\ra\C$ by
\begin{equation}\label{app-quadratic-form}
\QQ(\phi_{1},\phi_{2})=\int_{\La}\overline{\nabla\phi_{1}}\cdot\nabla\phi_{2}+\int_{\pa\La}\chi_{\La}\overline{\phi_{1}}\phi_{2},\quad \phi_{1},\phi_{2}\in H^{1}(\La).
\end{equation}
It is symmetric, closed and lower bounded. The corresponding self-adjoint operator, denoted by $-\De_{\La}^{\chi_{\La}}$, is the Laplacian with mixed $\chi_{\La}$-boundary conditions on $\pa\La$. Now, consider the periodic operator
\begin{equation*}
H_{\per}=-\De+V_{\per},
\end{equation*}
where $V_{\per}\in L^{p}_{\rm loc}(\R^{d})$ for $p>d$ is $\Z^{d}$-periodic. Then, we can define $H_{{\per},\La}^{\chi_{\La}}$ to be the operator $H_{\per}$ restricted to $L^{2}(\La)$ with mixed $\chi_{\La}$-boundary conditions on $\pa\La$. Moreover, the quadratic form \eqref{app-quadratic-form} corresponds to imposing Robin boundary condition $(\mathbf{n}_{\La}\cdot\nabla+\chi_{\La})\psi|_{\partial\La}=0$ for $\psi$ in the domain of the Laplacian on $L^{2}(\La)$, where $\mathbf{n}_{\La}:\pa\La\ra\R^{d}$ is the unit outer normal on $\pa\La$.

For the real-valued function $\chi_{\La}\in L^{\infty}(\pa\La)$, there's a very special choice. Let $E_{\per}$ be the ground state energy of $H_{\per}$ and $\vp_{\per}$ be the continuously differentiable, strictly positive ground state. Therefore, $\vp_{\per}$ is $\Z^{d}$-periodic, bounded from below by a positive constant and satisfies $H_{\per}\vp_{\per}=E_{\per}\vp_{\per}$. Define
\begin{equation}\label{Mezin-bd-fun}
\chi_{\La}(x)=-\frac{1}{\vp_{\per}(x)}\mathbf{n}_{\La}(x)\cdot\nabla\vp_{\per}(x),\quad x\in\partial\La.
\end{equation}
Since this choice of $\chi_{\La}$ was introduced by Mezincescu (see \cite{Me87}), it is called the Mezincescu boundary condition in his honor.

Main advantages of working with operators with Mezincescu boundary conditions are given by the following two lemmas.

\begin{lem}\label{lem-Mezin}
Let $\chi_{\La}\in L^{\infty}(\pa\La)$ be defined as in \eqref{Mezin-bd-fun} and denote $\chi_{\La}$ by $\chi_{L}$ if $\La=\La_{L}$ for $L\in\N$. Then, 
\begin{itemize}
\item[\rm(i)] $E_{\per}$ continuous to be the ground state energy of $H^{\chi_{\La}}_{{\per},\La}$;

\item[\rm(ii)] let $\vp=\vp_{\per}|_{\La}$. Then, $\vp$ is the strictly positive ground state of $H^{\chi_{\La}}_{{\per},\La}$, and hence, satisfies $H^{\chi_{\La}}_{{\per},\La}\vp=E_{\per}\vp$;
\end{itemize}
\end{lem}

Lemma \ref{lem-Mezin}$\rm(i)$ shows that the ground state energy of a periodic operator is inherited by its localized operators with Mezincescu boundary conditions defined via its ground state. This property is crucial here and it is not shared by Neumann operators unless in some special case, say, $V_{\per}$ is reflection symmetric (see Remark \ref{rem-1} for more details).

Mezincescu boundary condition also introduces the bracketing as Neumann bounday condition does. More precisely, for $i\in\Z^{d}$, let $\CC_{i}=i+\CC_{0}$ and denote by $\mathbf{n}_{i}:\partial\CC_{i}\ra\R^{d}$ the unit outer normal on $\partial\CC_{i}$, and define $\chi_{i}:\partial\CC_{i}\ra\R$ by setting
\begin{equation}\label{Mezin-bd-fun-i}
\chi_{i}(x)=-\frac{1}{\vp_{\per}(x)}\textbf{n}_{i}(x)\cdot\nabla\vp_{\per}(x),\quad x\in\partial\CC_{i}.
\end{equation}
We have

\begin{lem}\label{app-lem-bracketing}
Let $\chi_{\La}\in L^{\infty}(\pa\La)$ be defined as in \eqref{Mezin-bd-fun} and $\chi_{i}\in L^{\infty}(\pa\La)$ be defined as in \eqref{Mezin-bd-fun-i} for $i\in\Z^{d}$. Suppse that the side length of $\La$ is odd. Then
\begin{itemize}
\item[\rm(i)] $\chi_{i}=\chi_{\La}$ on $\pa\CC_{i}\cap\pa\La$ if $\pa\CC_{i}\cap\pa\La\neq\emptyset$;

\item[\rm(ii)] $\chi_{i_{1}}+\chi_{i_{2}}=0$ on $\pa\CC_{i_{1}}\cap\pa\CC_{i_{2}}$ for any $i_{1},i_{2}\in\Z^{d}$ with $\CC_{i_{1}}$ and $\CC_{i_{2}}$ being adjacent;

\item[\rm(iii)] there holds
\begin{equation*}
\lan\phi,-\De_{\La}^{\chi_{\La}}\phi\ran=\sum_{i\in\Z^{d}\cap\La}\lan\phi|_{\CC_{i}},-\De_{\CC_{i}}^{\chi_{i}}\phi|_{\CC_{i}}\ran,\quad\forall \phi\in H^{1}(\La).
\end{equation*}
In particular, the bracketing
\begin{equation*}
-\De_{\La}^{\chi_{\La}}\geq\bigoplus_{i\in\Z^{d}\cap\La}(-\De_{\CC_{i}}^{\chi_{i}})
\end{equation*}
is true in the sense of quadratic forms.
\end{itemize}
\end{lem}

\begin{rem}\label{rem-bracketing}
For Lemma \ref{app-lem-bracketing}, it is not necessary to use $\vp_{\per}$ in \eqref{Mezin-bd-fun} and \eqref{Mezin-bd-fun-i}. In fact, we can use any function defined on $\R^{d}$ that is real-valued, continuously differentiable, strictly positive and $\Z^{d}$-periodic.
\end{rem}

\begin{rem}
For notational simplicity, we will use suppressed notations for operators with Mezincescu boundary conditions in the sequel. More precisely, using the real-valued, continuously differentiable, strictly positive and $\Z^{d}$-periodic function $\vp$ to define the Mezincescu boundary condition on $\partial\La$, we need the function
\begin{equation*}
\chi_{\La}(x)=-\frac{1}{\vp(x)}\mathbf{n}_{\La}(x)\cdot\nabla\vp(x),\quad x\in\partial\La.
\end{equation*}
If $H$ is a self-adjoint operator on $L^{2}(\R^{d})$, we then denote by $H_{\La}^{\chi_{\La}}$ the operator $H$ restricted to $\La$ with Mezincescu boundary condition defined via $\vp$. Since $\La$ and $\vp$ determine $\chi_{\La}$, $H_{\La}^{\vp}$ may be a better notation. For simplicity, we will use $H_{\La}$ instead of $H_{\La}^{\chi_{\La}}$ or $H_{\La}^{\vp}$.
\end{rem}


\subsection{Determining the Bottom of the Spectrum}\label{sec-spectral-minimum}

This subsection is devoted to the characterization of $E_{0}$, that is, we will prove Theorem \ref{thm-spectral-minimum}. Recall that
\begin{equation}\label{operator-single}
H_{t}=-\De+V_{0}+t\sum_{i\in\Z^{d}}u(\cdot-i),\quad t\in[a,b].
\end{equation}
Let $\vp\in C^{1}(\R^{d})$ be real-valued, strictly positive and $\Z^{d}$-periodic. Denote by $H_{t,\CC_{0}}$ the restriction of $H_{t}$ to $L^{2}(\CC_{0})$ with Mezincescu boundary condition defined via $\vp$ on $\partial\CC_{0}$ and by $E_{\vp}(t)$ the ground state energy of $H_{t,\CC_{0}}$.

To prove Theorem \ref{thm-spectral-minimum}, we first establish some lemmas.

\begin{lem}\label{lem-ground-state}
Suppose $\rm (H1)$, $\rm(H2)$ and $\rm(H3)$. Let $\vp\in C^{1}(\R^{d})$ be real-valued, strictly positive and $\Z^{d}$-periodic. Then,
\begin{itemize}
\item[\rm(i)] $E_{\vp}(\cdot)$ is real analytic and strictly concave on $[a,b]$;

\item[\rm(ii)] the bottom of $\Si$, i.e., $E_{0}$, satisfies $E_{0}\geq\min\{E_{\vp}(a),E_{\vp}(b)\}$.
\end{itemize}
\end{lem}
\begin{proof}
$\rm(i)$ The real analyticity follows from analytic perturbation theory (see e.g. \cite{RS78}). For $t\in[a,b]$, define the functional $E_{\vp}(\cdot,t):H^{1}(\CC_{0})\ra\R$ by
\begin{equation*}
E_{\vp}(\phi,t)=\|\nabla\phi\|^{2}+\int_{\pa\CC_{0}}\chi_{0}|\phi|^{2}+\int_{\CC_{0}}V_{0}|\phi|^{2}+t\int_{\CC_{0}}u|\phi|^{2},\quad\phi\in H^{1}(\CC_{0}),
\end{equation*}
and then, $E_{\vp}(t)=\inf_{\phi\in H^{1}(\CC_{0}),\|\phi\|=1}E_{\vp}(\phi,t)$. The concavity then follows directly from the fact that $E_{\vp}$ is the infimum of an affine function.

For the strict concavity, we use the following identity
\begin{equation}\label{identity-1}
E_{\vp}''(t)=-2\sum_{n=1}^{\infty}\frac{\lan u\vp_{0}(H_{t,\CC_{0}}),\vp_{n}(H_{t,\CC_{0}})\ran^{2}}{E_{n}(H_{t,\CC_{0}})-E_{0}(H_{t,\CC_{0}})}=-2\sum_{n=1}^{\infty}\frac{\lan u\vp_{0}(H_{t,\CC_{0}}),\vp_{n}(H_{t,\CC_{0}})\ran^{2}}{E_{n}(H_{t,\CC_{0}})-E_{\vp}(t)},
\end{equation}
where $\vp_{n}(H_{t,\CC_{0}})$, $n\in\N_{0}$ are real normalized eigenfunctions corresponding to $E_{n}(H_{t,\CC_{0}})$, $n\in\N_{0}$. \eqref{identity-1} is proven in \cite[Eq.(11)]{BLS08} for Neumann operators and the proof there is applied in our situation. Using \eqref{identity-1}, we conclude from the simplicity of the ground state energy that $E_{\vp}''(t)<0$ for all $t\in[a,b]$ unless $u$ is a constant function, which, however, is excluded by our assumption. Hence, $E_{\vp}(\cdot)$ is strictly concave on $[a,b]$.

$\rm(ii)$ Recall $\CC_{i}=i+\CC_{0}$ for $i\in\Z^{d}$. By Lemma \ref{app-lem-bracketing} and Remark \ref{rem-bracketing}, we have the bracketing $H_{\om}\geq\bigoplus_{i\in\Z^{d}}H_{\om,\CC_{i}}$. Due to the $\Z^{d}$-periodicity of $V_{0}$ and $\vp$, the operator $H_{\om,\CC_{i}}$ is unitarily equivalent to $H_{\om_{i},\CC_{0}}$ for every $i\in\Z^{d}$, which leads to
\begin{equation}\label{bracketing}
E_{0}\geq\inf_{t\in\supp(\P_{0})}E_{\vp}(t).
\end{equation}
The result then follows from \eqref{bracketing} and $\rm(i)$.
\end{proof}

\begin{lem}\label{lem-ground-state-upper-bound}
Suppose $\rm (H1)$, $\rm(H2)$ and $\rm(H3)$. Let $t\in\supp(\P_{0})$. Then
\begin{equation*}
\min\big\{E_{\vp_{t}}(a),E_{\vp_{t}}(b)\big\}\leq E_{0}\leq E_{\vp_{t}}(t),
\end{equation*}
where $\vp_{t}$ is the continuously differentiable, strictly positive and $\Z^{d}$-periodic ground state of $H_{t}$.
\end{lem}
\begin{proof}
By Lemma \ref{lem-ground-state}, it suffices to prove the second inequality. By periodic approximation, we have $\si(H_{t})\subset\Si$, and hence, $E_{0}\leq\inf\si(H_{t})$. Since $\vp_{t}$ is the continuously differentiable, strictly positive ground state of $H_{t}$, we conclude from Lemma \ref{lem-Mezin} that $\inf\si(H_{t})$ is also the ground state energy of $H_{t,\CC_{0}}$. Thus, we have $E_{0}\leq\inf\si(H_{t})=E_{\vp_{t}}(t)$.
\end{proof}

Theorem \ref{thm-spectral-minimum} now follows directly from Lemma \ref{lem-ground-state-upper-bound}.

\begin{proof}[Proof of Theorem \ref{thm-spectral-minimum}]
Setting $t=a$ and $t=b$, respectively, in Lemma \ref{lem-ground-state-upper-bound}, we find $\rm(i)$ and $\rm(ii)$. $\rm(iii)$ is a consequence of Lemma \ref{lem-ground-state-upper-bound} and the continuity of $E_{\vp_{a}}(\cdot)$ and $E_{\vp_{b}}(\cdot)$ by Lemma \ref{lem-ground-state}$\rm(i)$.
\end{proof}

\begin{rem}\label{rem-1}
\begin{itemize}
\item[\rm(i)]
The characterization of the bottom of the spectrum for alloy type models with sign-indefinite single-site potentials was first studied in \cite{Na06} by Najar for sufficiently small $a$ and $b$ with an additional assumption on the sign of $\int_{\R^{d}}u(x)dx$. Later, Klopp and Nakamura proved in \cite{KN09} the same result with a reflection symmetry assumption \eqref{reflection-symmetric} on $u$ as mentioned before. Also, there are corresponding results for random displacement models (see e.g. \cite{BLS08,BLS09}).

\item[\rm(ii)] Reviewing the proof of Theorem \ref{thm-spectral-minimum}, it's easy to see that the arguments are completely based on Lemma \ref{lem-Mezin} $\rm(i)$ and Lemma \ref{app-lem-bracketing} $\rm(iii)$, which actually generalize corresponding results in the case of Neumann operators with $V_{\per}$ being reflection symmetric. In fact, it is well-known (see e.g. \cite{RS78}) that Neumann operators enjoy the bracketing as in Lemma \ref{app-lem-bracketing} $\rm(iii)$. Moreover, we claim that $\inf\si(H_{\per})=\inf\si(H_{{\per},\CC_{0}}^{N})$. Indeed, $\inf\si(H_{\per})=\inf\si(H_{{\per},\CC_{0}}^{P})$ by Floquet theory. Since $H_{{\per},\CC_{0}}^{N}\leq H_{{\per},\CC_{0}}^{P}$, $\inf\si(H_{{\per},\CC_{0}}^{N})\leq \inf\si(H_{{\per},\CC_{0}}^{P})$. But the reflection symmetry of $V_{\per}$ yields that the ground state corresponding to $\inf\si(H_{{\per},\CC_{0}}^{N})$ satisfies periodic boundary condition, and hence, $\inf\si(H_{{\per},\CC_{0}}^{N})\in\si(H_{{\per},\CC_{0}}^{P})$, which leads to $\inf\si(H_{{\per},\CC_{0}}^{N})\geq \inf\si(H_{{\per},\CC_{0}}^{P})$. Thus, $\inf\si(H_{{\per},\CC_{0}}^{N})=\inf\si(H_{{\per},\CC_{0}}^{P})=\inf\si(H_{\per})$.
\end{itemize}
\end{rem}

\subsection{The Integrated Density of States}\label{subsec-IDS}

For $E\in\R$, define the eigenvalue counting function
\begin{equation*}
N\big(H_{\om,\La_{L}}^{X},E\big)=\#\big\{n\in\N_{0}\big|E_{n}(H_{\om,\La_{L}}^{X})\leq E\big\},
\end{equation*}
where $X=D$ and $N$ refer to Dirichlet and Neumann boundary conditions, respectively. Under the assumptions $\rm(H1)$, $\rm(H2)$ and $\rm(H3)$, for $E\in\R$, the limit
\begin{equation*}
N^{X}(E):=\lim_{L\rightarrow\infty}\frac{N(H_{\om,\La_{L}}^{X},E)}{L^{d}}
\end{equation*}
exists and a.e. deterministic. Moreover, $N^{D}(E)=N^{N}(E)$ for all but possible countably many $E\in\R$. Their common value is called the integrated density of states, denoted by $N(E)$, $E\in\R$, of $H_{\om}$. See \cite{Ve08} for the proof. Also, for $E\in\R$,
\begin{equation}\label{IDS-expectation-X}
N(E)=\sup_{L\in\N}\frac{\E\{N(H_{\cdot,\La_{L}}^{D},E)\}}{L^{d}}=\inf_{L\in\N}\frac{\E\{N(H_{\cdot,\La_{L}}^{N},E)\}}{L^{d}}
\end{equation}
Our objective is to investigate the asymptotic behavior of $N(E)$ near $E_{0}$. More precisely, we wish \begin{equation*}
\lim_{E\da E_{0}}\frac{\ln|\ln N(E)|}{\ln(E-E_{0})}=-\frac{d}{2}.
\end{equation*}

To prove Theorem \ref{thm-upper-bound}, Theorem \ref{thm-Lifshitz-tail-non-critical} and Theorem \ref{thm-Lifshitz-tail-non}, we need the IDS to be defined via eigenvalue counting functions of operators with Mezincescu boundary conditions. More precisely, if we use $\vp$, a continuously differentiable, real-valued, strictly positive and $\Z^{d}$-periodic function, to define Mezincescu boundary conditions,  then under the additional assumption $\rm(H4)$, we have
\begin{equation*}
N(E)=\lim_{L\ra\infty}\frac{N(H_{\om,\La_{L}},E)}{L^{d}},\quad E\in\R,
\end{equation*}
where $N(H_{\om,\La_{L}},\cdot)$ is the eigenvalue value counting function of $H_{\om,\La_{L}}$. Due to Lemma \ref{app-lem-bracketing} and Remark \ref{rem-bracketing}, we also have
\begin{equation}\label{IDS-expectation}
N(E)=\inf_{L\in\N}\frac{\E\{N(H_{\cdot,\La_{L}},E)\}}{L^{d}},\quad E\in\R.
\end{equation}


\section{Lifshitz Tails: Optimal Upper Bound}\label{sec-lifshitz-tail}

We prove Theorem \ref{thm-upper-bound} in this section. By symmetry, we focus on the case $E_{\vp_{a}}(a)<E_{\vp_{a}}(b)$. Therefore, 
\begin{itemize}
\item[$\bullet$] $\rm (H1)$, $\rm(H2)$, $\rm(H3)$, $\rm(H4)$ and $E_{\vp_{a}}(a)<E_{\vp_{a}}(b)$
\end{itemize}
is always assumed in this section. Also, all the Mezincescu boundary conditions in this section are defined via $\vp_{a}$.

The following lemma is the key to the proof of Theorem \ref{thm-upper-bound}. Its proof is based on the operator theoretical trick (a comparison method) developed in \cite[Theorem 2.1, Lemma 2.1]{KN09} by Klopp and Nakamura.

\begin{lem}\label{lem-upper-bound-1}
There exists some $C>0$ such that
\begin{equation*}
N(E)\leq N_{a}(C(E-E_{\vp_{a}}(a))),\,\, E\in\R,
\end{equation*}
where $N_{a}$ is the IDS of
\begin{equation*}
H_{a,\om}=H_{a}-E_{\vp_{a}}(a)+\sum_{i\in\Z^{d}}(\om_{i}-a)1_{\CC_{0}}(\cdot-i).
\end{equation*}
\end{lem}
\begin{proof}
Let $E\in\R$. By \eqref{IDS-expectation}, to show $N(E)\leq N_{a}(C(E-E_{\vp_{a}}(a)))$, it suffices to show that for large $L\in2\N_{0}+1$
\begin{equation*}
N(H_{\om,\La_{L}},E)\leq N(H_{a,\om,\La_{L}},C(E-E_{\vp_{a}}(a))),
\end{equation*}
which is true if the operator inequality $H_{a,\om,\La_{L}}\leq C(H_{\om,\La_{L}}-E_{\vp_{a}}(a))$ holds in the sense of quadratic form, that is,
\begin{equation}\label{quadratic-form-inequality}
\lan \phi,H_{a,\om,\La_{L}}\phi\ran\leq C\lan\phi,(H_{\om,\La_{L}}-E_{\vp_{a}}(a))\phi\ran,\quad\forall\phi\in H^{1}(\La_{L}).
\end{equation}
Using Lemma \ref{app-lem-bracketing} and Remark \ref{rem-bracketing}, we find for any $\phi\in H^{1}(\La_{L})$
\begin{equation*}
\lan \phi,H_{a,\om,\La_{L}}\phi\ran=\sum_{i\in\Z^{d}\cap\La_{L}}\lan\phi|_{\CC_{i}},H_{a,\om,\CC_{i}}\phi|_{\CC_{i}}\ran
\end{equation*}
and
\begin{equation*}
\lan\phi,(H_{\om,\La_{L}}-E_{\vp_{a}}(a))\phi\ran=\sum_{i\in\Z^{d}\cap\La_{L}}\lan\phi|_{\CC_{i}},(H_{\om,\CC_{i}}-E_{\vp_{a}}(a))\phi|_{\CC_{i}}\ran.
\end{equation*}
Therefore, to show \eqref{quadratic-form-inequality}, it suffices to require that for all $i\in\Z^{d}\cap\La_{L}$
\begin{equation}\label{quadratic-form-inequalities}
\lan\phi,H_{a,\om,\CC_{i}}\phi\ran\leq C\lan\phi,(H_{\om,\CC_{i}}-E_{\vp_{a}}(a))\phi\ran,\quad\forall\phi\in H^{1}(\CC_{i}).
\end{equation}
By the $\Z^{d}$-periodicity of $V_{0}$ and $\vp$, for any $i\in\Z^{d}\cap\La_{L}$, $H_{a,\om,\CC_{i}}$ and $H_{\om,\CC_{i}}-E_{\vp_{a}}(a)$ are unitarily equivalent to $H_{a,\CC_{0}}-E_{\vp_{a}}(a)+\om_{i}-a$ and $H_{\om_{i},\CC_{0}}-E_{\vp_{a}}(a)$, respectively, which will lead to \eqref{quadratic-form-inequalities} if we can show that
\begin{equation}\label{operator-inequality}
H_{a,\CC_{0}}-E_{\vp_{a}}(a)+t-a\leq C(H_{t,\CC_{0}}-E_{\vp_{a}}(a)),\quad\forall t\in[a,b].
\end{equation}

To finish the proof, we show \eqref{operator-inequality}. Since
\begin{equation*}
\inf\si(H_{b,\CC_{0}}-E_{\vp_{a}}(a))=E_{\vp_{a}}(b)-E_{\vp_{a}}(a)>0
\end{equation*}
by assumption, regular perturbation theory (see e.g. \cite{RS78}) ensures that we can find some $\beta>b$ and $\de>0$ such that $\inf\si(H_{\beta,\CC_{0}}-E_{\vp_{a}}(a))=\de$. It then follows that for any $t\in[a,b]$
\begin{equation*}
\begin{split}
&H_{t,\CC_{0}}-E_{\vp_{a}}(a)\\
&\quad\quad=H_{a,\CC_{0}}-E_{\vp_{a}}(a)+(t-a)u\\
&\quad\quad=\bigg(1-\frac{t-a}{\beta-a}\bigg)(H_{a,\CC_{0}}-E_{\vp_{a}}(a))+\frac{t-a}{\beta-a}(H_{a,\CC_{0}}-E_{\vp_{a}}(a)+(\beta-a)u)\\
&\quad\quad\geq\bigg(1-\frac{b-a}{\beta-a}\bigg)(H_{a,\CC_{0}}-E_{\vp_{a}}(a))+\frac{\de}{\beta-a}(t-a)\\
&\quad\quad\geq\frac{1}{C}(H_{a,\CC_{0}}-E_{\vp_{a}}(a)+t-a),
\end{split}
\end{equation*}
where $\frac{1}{C}=\min\big\{1-\frac{b-a}{\beta-a},\frac{\de}{\beta-a}\big\}$. This completes the proof.
\end{proof}

We now prove Theorem \ref{thm-upper-bound}.

\begin{proof}[Proof of Theorem \ref{thm-upper-bound}]
By means of Lemma \ref{lem-upper-bound-1}, to prove the upper bound, it suffices to estimate an appropriate upper bound for $N_{a}$. To do so, we set $H_{a,\per}=H_{a}-E_{\vp_{a}}(a)$, that is, the periodic part of $H_{a,\om}$. Clearly, $\inf\si(H_{a,\per})=0$ and the ground state of $H_{a,\per}$ is the same as that of $H_{a}$.

For upper bound for $N_{a}$, we claim that there exist $C_{1}>0$ and $C_{2}>0$ such that for all $E\in\R$ and all large $L\in2\N_{0}+1$
\begin{equation*}\label{upper-bound-estimate}
N_{a}(E)\leq\frac{1}{L^{d}}N(H_{a,{\per},\La_{L}},E)\P(\Om_{a,L,E}),
\end{equation*}
where $N(H_{a,{\per},\La_{L}},\cdot)$ is the eigenvalue counting function of $H_{a,{\per},\La_{L}}$ and
\begin{equation*}
\Om_{a,L,E}=\bigg\{\om\in\Om\bigg|\frac{\#\big\{i\in\Z^{d}\cap\La_{L}\big|\om_{i}-a<C_{1}L^{-2}\big\}}{L^{d}}>C_{2}L^{2}E\bigg\}.
\end{equation*}
Indeed, since $N(H_{a,\om,\La_{L}},E)=0$ for $E<E_{0}(H_{a,\om,\La_{L}})$, we have
\begin{equation*}
\begin{split}
\E\{N(H_{a,\cdot,\La_{L}},E)\}&=\int_{\big\{\om\in\Om\big|E_{0}(H_{a,\om,\La_{L}})\leq E\big\}}N(H_{a,\om,\La_{L}},E)d\P(\om)\\
&\leq N(H_{a,{\per},\La_{L}},E)\P\big\{\om\in\Om\big|E_{0}(H_{a,\om,\La_{L}})\leq E\big\},
\end{split}
\end{equation*}
where we used the fact $H_{a,\om,\La_{L}}\geq H_{a,\per,\La_{L}}$ such that $N(H_{a,\om,\La_{L}},E)\leq N(H_{a,{\per},\La_{L}},E)$. The same reason for \eqref{IDS-expectation} implies that
\begin{equation*}
N_{a}(E)\leq\frac{1}{L^{d}}N(H_{a,{\per},\La_{L}},E)\P\big\{\om\in\Om\big|E_{0}(H_{a,\om,\La_{L}})\leq E\big\}.
\end{equation*}
The estimate $\P\big\{\om\in\Om\big|E_{0}(H_{a,\om,\La_{L}})\leq E\big\}\leq\P(\Om_{a,L,E})$, following from Temple's inequality (see e.g. \cite[Theorem XIII.5]{RS78}), is standard. We refer to \cite{KW05} for more details.

Considering  the van-Hove singularity (see e.g. \cite{KS87}) of the IDS of $H_{a,\rm per}$ near $0$, the theorem is a consequence of Lemma \ref{lem-upper-bound-1} and the above claim with a large deviation argument (see e.g. \cite{Ki08}).
\end{proof}

\section{Lower Bound of Ground State Energy}\label{sec-lower-bound-gse}

This section serves as a preparation for proofs of Theorem \ref{thm-Lifshitz-tail-non-critical} and Theorem \ref{thm-Lifshitz-tail-non}, which will be given in Section \ref{sec-LT-non-optimal-upper-bound}. Thus, we treat the problem under $\rm(i)$ $E_{\vp_{a}}(a)=E_{\vp_{a}}(b)$ and $E_{\vp_{b}}(a)\leq E_{\vp_{b}}(b)$, or $\rm(ii)$ $E_{\vp_{a}}(a)\geq E_{\vp_{a}}(b)$ and $E_{\vp_{b}}(a)=E_{\vp_{b}}(b)$. By Theorem \ref{thm-spectral-minimum}, $E_{0}=E_{\vp_{a}}(a)=E_{\vp_{a}}(b)$ if $\rm(i)$ is satisfied, and $E_{0}=E_{\vp_{b}}(a)=E_{\vp_{b}}(b)$ if $\rm(ii)$ is satisfied.

To fix the ideal, we focus on $\rm(i)$. Also, to simplify statements, we always assume
\begin{itemize}
\item[$\bullet$] $\rm (H1)$, $\rm(H2)$, $\rm(H3)$, $\rm(H4)$, $\rm(H5)$, $\rm(H6)$, $E_{\vp_{a}}(a)=E_{\vp_{a}}(b)$ and $E_{\vp_{b}}(a)\leq E_{\vp_{b}}(b)$.
\end{itemize}
Therefore, all the Mezincescu boundary conditions in this section are defined via $\vp_{a}$.

We point out that if $E_{\vp_{a}}(a)=E_{\vp_{a}}(b)$ and $E_{\vp_{b}}(a)=E_{\vp_{b}}(b)$, then $\rm(H6)$ is not required, in fact, $\rm(H6)$ is always the case (see Lemma \ref{lem-key-2} below).

To state the main result in this section, we set
\begin{equation}\label{aux-domain}
\begin{split}
\Om_{0}&=\bigg(-\frac{1}{2},\frac{1}{2}\bigg)^{d-1}\times\bigg(-\frac{m}{2},-\frac{1}{2}\bigg),\quad m\in2\N_{0}+3,\\
\Om_{M}&=\bigg(-\frac{1}{2},\frac{1}{2}\bigg)^{d-1}\times\bigg(-\frac{1}{2},\frac{M}{2}\bigg),\quad M\in2\N_{0}+1,\\
\Om_{0M}&=\text{int}(\ol{\Om_{0}\cup\Om_{M}})
\end{split}
\end{equation}
and consider the operator
\begin{equation*}
-\De_{\Om_{0M}}+V_{0}1_{\Om_{0M}}+W_{\Om_{0M}}\quad\text{on}\quad\Om_{0M},
\end{equation*}
where the potential $W_{\Om_{0M}}$ is defined as follows: $W_{\Om_{0M}}1_{\Om_{0}}$ is of the form $\sum_{i\in\Z^{d}\cap\Om_{0}}\om_{i}u(\cdot-i)$ and $W_{\Om_{0M}}1_{\Om_{M}}=a\sum_{i\in\Z^{d}\cap\Om_{M}}u(\cdot-i)$ or $b\sum_{i\in\Z^{d}\cap\Om_{M}}u(\cdot-i)$. Our goal is to prove

\begin{thm}\label{thm-lower-bound-gse-main}
If $\inf\si(-\De_{\Om_{0}}+V_{0}1_{\Om_{0}}+W_{\Om_{0M}}1_{\Om_{0}})>E_{0}$, then there's some $M$-independent $C>0$ such that
\begin{equation*}
\inf\si(-\De_{\Om_{0M}}+V_{0}1_{\Om_{0M}}+W_{\Om_{0M}})\geq E_{0}+\frac{C}{M^{2}}
\end{equation*}
for all $M\in2\N_{0}+1$.
\end{thm}

We remark that the constant $C$ in Theorem \ref{thm-lower-bound-gse-main} does depend on $m\in2\N_{0}+3$ and the potential $W_{\Om_{0M}}1_{\Om_{0}}$. We here do not make this dependence clear for the reason that it will not play a role when we apply Theorem \ref{thm-lower-bound-gse-main} in Section \ref{sec-LT-non-optimal-upper-bound} (see Remark \ref{rem-lower-bound-gse}$\rm(ii)$ for more details).

We also need a result with $\Om_{0}$ above $\Om_{M}$ in terms of the $d$-th coordinate. Since its proof is the same as that of Theorem \ref{thm-lower-bound-gse-main}, we will only state it in Theorem \ref{thm-lower-bound-gse-main-1}.

Due to technical reasons, the proof of Theorem \ref{thm-lower-bound-gse-main} will be separated according to the cases $W_{\Om_{0M}}1_{\Om_{M}}=a\sum_{i\in\Z^{d}\cap\Om_{M}}u(\cdot-i)$ (in Subsection \ref{subsec-inherited-gse}) and $W_{\Om_{0M}}1_{\Om_{M}}=b\sum_{i\in\Z^{d}\cap\Om_{M}}u(\cdot-i)$ (in Subsection \ref{subsec-non-inherited-gse}). See Theorem \ref{thm-lower-bound-estimate} and Theorem \ref{thm-lower-bound-estimate-1} below.

\subsection{Inherited Ground State Energy}\label{subsec-inherited-gse}

We prove Theorem \ref{thm-lower-bound-gse-main} in the case $W_{\Om_{0M}}1_{\Om_{M}}=a\sum_{i\in\Z^{d}\cap\Om_{M}}u(\cdot-i)$. Consider operators defined as follows: let $\Om_{0}=(-\frac{1}{2},\frac{1}{2})^{d-1}\times(-\frac{m}{2},-\frac{1}{2})$ with $m\in2\N_{0}+3$ and set
\begin{equation*}
P_{0}=-\De_{\Om_{0}}+W_{0},
\end{equation*}
where $W_{0}$ is of the form $V_{0}1_{\Om_{0}}+\sum_{i\in\Z^{d}\cap\Om_{0}}\om_{i}u(\cdot-i)$; let $\Om_{M}=(-\frac{1}{2},\frac{1}{2})^{d-1}\times(-\frac{1}{2},\frac{M}{2})$ with $M\in2\N_{0}+1$ and set
\begin{equation*}
H_{a,\Om_{M}}=-\De_{\Om_{M}}+W_{M}=-\De_{\Om_{M}}+V_{0}1_{\Om_{M}}+a\sum_{i\in\Z^{d}\cap\Om_{M}}u(\cdot-i);
\end{equation*}
let $\Om_{0M}=\text{int}(\ol{\Om_{0}\cup\Om_{M}})$ and set
\begin{equation*}
P_{0M}=-\De_{\Om_{0M}}+W_{0M},
\end{equation*}
where $W_{0M}=W_{0}1_{\Om_{0}}+W_{M}1_{\Om_{M}}$.

Note $H_{a,\Om_{M}}$ is the operator $H_{a}$ restricted to $\Om_{M}$ with Mezincescu boundary condition. Since Mezincescu boundary condition is defined via $\vp_{a}$, the ground state energy and the ground state of $H_{a,\Om_{M}}$ are inherited from that of $H_{a}$, that is, $\inf\si(H_{a,\Om_{M}})=\inf\si(H_{a})=E_{0}$ and the ground state of $H_{a,\Om_{M}}$ is nothing but $\vp_{a}$ restricted to $\Om_{M}$.

Theorem \ref{thm-lower-bound-gse-main} in the case $W_{\Om_{0M}}1_{\Om_{M}}=a\sum_{i\in\Z^{d}\cap\Om_{M}}u(\cdot-i)$ is restated as

\begin{thm}\label{thm-lower-bound-estimate}
If $\inf\si(P_{0})>E_{0}$, then there exists some $M$-independent constant $C>0$ such that
\begin{equation*}
\inf\si(P_{0M})\geq E_{0}+\frac{C}{M^{2}}
\end{equation*}
for all $M\in2\N_{0}+1$.
\end{thm}

To prove the above theorem, we adapt the quasi one-dimensional estimate developed by Klopp and Nakamura (see \cite{KN10}). We begin with several lemmas.

Set $S=(-\frac{1}{2},\frac{1}{2})^{d-1}\times\{-\frac{1}{2}\}$ and define the trace operator $\Ga_{M}:C^{0}(\Om_{M})\ra C^{0}(S)$ by
\begin{equation*}
(\Ga_{M}\psi)(x')=\psi(x',-\frac{1}{2}),\quad x'=(x_{1},\dots,x_{d-1})
\end{equation*}
for $\psi\in C^{0}(\Om_{M})$. $\Ga_{M}$ is then extended to a bounded linear operator from $H^{1}(\Om_{M})$ to $L^{2}(S)$. Following  \cite[Lemma 2.2]{KN10}, we have the variant of the classical Poincar\'{e} inequality
\begin{equation}\label{estimate-1}
\frac{4}{M}\|\Ga_{M}\psi\|^{2}+\|\nabla\psi\|^{2}\geq\frac{4}{M(M+1)}\|\psi\|^{2},\quad \psi\in H^{1}(\Om_{M})
\end{equation}
for all $M\in2\N_{0}+1$.

\begin{lem}\label{lem-estimate-1}
For any $\psi\in H^{1}(\Om_{M})$, there holds
\begin{equation*}
\frac{4}{M}\|\Ga_{M}\psi\|^{2}+Q_{M}(\psi,\psi)-E_{0}\|\psi\|^{2}\geq\bigg(\frac{\inf\vp_{a}}{\sup\vp_{a}}\bigg)^{2}\frac{4}{M(M+1)}\|\psi\|^{2},
\end{equation*}
where $Q_{M}(\cdot,\cdot)$ is the quadratic form of $H_{a,\Om_{M}}$.
\end{lem}
\begin{proof}
Denoted by $\vp_{a,M}=\vp_{a}|_{\Om_{M}}$ the ground state of $H_{a,\Om_{M}}$. We have
\begin{equation}\label{estimate-2}
\bigg\|\nabla\bigg(\frac{\psi}{\vp_{a,M}}\bigg)\bigg\|^{2}\leq\frac{1}{(\inf\vp_{a})^{2}}\big[Q_{M}(\psi,\psi)-E_{0}\|\psi\|^{2}\big],\quad\psi\in H^{1}(\Om_{M}).
\end{equation}
This follows from the arguments as in the proof of \cite[Lemma 2.3]{KN10} by means of the ground state transform. Here, one more boundary term is involved, but this does not cause any trouble.

For $\psi\in H^{1}(\Om_{M})$, we apply \eqref{estimate-1} with $\frac{\psi}{\vp_{a,M}}$ and use \eqref{estimate-2} to obtain
\begin{equation*}
\begin{split}
\frac{4}{M(M+1)}\|\psi\|^{2}&\leq(\sup\vp_{a})^{2}\frac{4}{M(M+1)}\bigg\|\frac{\psi}{\vp_{a,M}}\bigg\|^{2}\\
&\leq\bigg(\frac{\sup\vp_{a}}{\inf\vp_{a}}\bigg)^{2}\frac{4}{M}\|\Ga_{M}\psi\|^{2}+\big(\sup\vp_{a}\big)^{2}\bigg\|\nabla\bigg(\frac{\psi}{\vp_{a,M}}\bigg)\bigg\|^{2}\\
&\leq\bigg(\frac{\sup\vp_{a}}{\inf\vp_{a}}\bigg)^{2}\frac{4}{M}\|\Ga_{M}\psi\|^{2}+\bigg(\frac{\sup\vp_{a}}{\inf\vp_{a}}\bigg)^{2}\big[Q_{M}(\psi,\psi)-E_{0}\|\psi\|^{2}\big],
\end{split}
\end{equation*}
which leads to the result.
\end{proof}

Set $\al=\inf\si(P_{0})$. Suppose $\al>E_{0}$ as in Theorem \ref{thm-lower-bound-estimate}. Let $\Ga_{\Om_{0}}:H^{1}(\Om_{0})\ra L^{2}(S)$ be the trace operator defined in the same way as that of $\Ga_{M}$. Let $\la<\al$. As in \cite{KN10}, we consider the following eigenvalue problems
\begin{equation}\label{eigen-problem}
\begin{split}
\left\{\begin{aligned}
&(-\De+W_{0})\psi=\la\psi\quad\text{in}\,\,\Om_{0},\\
&\Ga_{\Om_{0}}\psi=g\in H^{3/2}(S),\\
&(\textbf{n}_{\Om_{0}}\cdot\nabla+\chi_{\Om_{0}})\psi=0\quad\text{on}\,\,\pa\Om_{0}\backslash S.
\end{aligned} \right.
\end{split}
\end{equation}
By standard arguments of the theory of elliptic boundary value problems (see e.g. \cite{Fo95}), the eigenvalue problem \eqref{eigen-problem} have a unique solution $\psi\in H^{2}(\Om_{0})$ for any $g\in H^{3/2}(S)$. Moreover,
\begin{equation*}
T_{\la}:H^{3/2}(S)\ra H^{1/2}(S),\quad g\mapsto\Ga_{\Om_{0}}(\pa_{d}\psi)
\end{equation*}
defines a bounded linear operator.

\begin{lem}\label{lem-estimate}
\begin{itemize}
\item[\rm(i)] $T_{\la}$ is symmetric for any $\la<\al$.
\item[\rm(ii)] Let $\la_{0}\in[E_{0},\al)$. There's some $\ep=\ep(\la_{0})$ such that
\begin{equation*}
\lan g,T_{\la}g\ran+\int_{S}\chi_{\Om_{0}}|g|^{2}\geq\ep\|g\|^{2},\quad g\in H^{3/2}(S)
\end{equation*}
for all  $\la\in[E_{0},\la_{0}]$.
\end{itemize}
\end{lem}
\begin{proof}
$\rm(i)$ is a simple consequence of Green's formula. To verify $\rm(ii)$, we let $g\in H^{3/2}(S)$ and $\psi\in H^{2}(\Om_{0})$ be the unique solution of \eqref{eigen-problem}. Thus, $\Ga_{\Om_{0}}\psi=g$, $T_{\la}g=\Ga_{\Om_{0}}(\pa_{d}\psi)$ and
\begin{equation*}
(-\De+W_{0})\psi=\la\psi\quad\text{in}\,\,\Om_{0}.
\end{equation*}
Using Green's formula, we calculate
\begin{equation*}
\begin{split}
0&=\lan\psi,(-\De+W_{0}-\la)\psi\ran\\
&=\int_{\Om_{0}}|\nabla\psi|^{2}-\int_{\pa\Om_{0}\backslash S}\ol{\psi}(\textbf{n}_{\Om_{0}}\cdot\nabla\psi)-\int_{S}\ol{\psi}(\pa_{d}\psi)+\int_{\Om_{0}}(W_{0}-\la)|\psi|^{2}\\
&=\int_{\Om_{0}}|\nabla\psi|^{2}+\int_{\pa\Om_{0}\backslash S}\chi_{\Om_{0}}|\psi|^{2}-\int_{S}\ol{g}T_{\la}g+\int_{\Om_{0}}(W_{0}-\la)|\psi|^{2},
\end{split}
\end{equation*}
which leads to
\begin{equation*}
\lan g,T_{\la}g\ran+\int_{S}\chi_{\Om_{0}}|\psi|^{2}=Q_{\Om_{0}}(\psi,\psi)-\la\|\psi\|^{2}\geq Q_{\Om_{0}}(\psi,\psi)-\la_{0}\|\psi\|^{2},
\end{equation*}
where $Q_{\Om_{0}}(\cdot,\cdot)$ is the quadratic form of $P_{0}$. Since $Q_{\Om_{0}}\geq\al>\la_{0}$, $Q_{\Om_{0}}-\la_{0}$ is strictly positive with form domain $H^{1}(\Om_{0})$ and the corresponding strictly positive self-adjoint operator is given by $P_{0}-\la_{0}$. Moreover, the domain of  $\sqrt{P_{0}-\la_{0}}$ is $H^{1}(\Om_{0})$. The strict positivity of $P_{0}-\la_{0}$ implies the equivalence of the norms $\|\sqrt{P_{0}-\la_{0}}\cdot\|$ and $\|\cdot\|_{H^{1}(\Om_{0})}$ on $H^{1}(\Om_{0})$. It then follows that
\begin{equation*}
\begin{split}
Q_{\Om_{0}}(\psi,\psi)-\la_{0}\|\psi\|^{2}&=\|\sqrt{P_{0}-\la_{0}}\psi\|^{2}\geq C\|\psi\|_{H^{1}(\Om_{0})}^{2}\geq\ep\|\Ga_{\Om_{0}}\psi\|^{2}
\end{split}
\end{equation*}
for some $C,\ep>0$, where the last inequality is due to the boundedness of the trace operator. This completes the proof.
\end{proof}

We now prove Theorem \ref{thm-lower-bound-estimate}.

\begin{proof}[Proof of Theorem \ref{thm-lower-bound-estimate}]
Fix any $M\in2\N_{0}+1$. Let $\psi_{0M}$ be the strictly positive ground state of $P_{0M}$ with the ground state energy $\la_{0M}$. Since $\inf\si(P_{0})>E_{0}$ and $\inf\si(H_{a,\Om_{M}})=E_{0}$, we conclude from Lemma \ref{app-lem-bracketing} that $\la_{0M}\geq E_{0}$.

We first prove that the theorem holds for all not-very-large $M$. To do so, it suffice to show $\la_{0M}>E_{0}$ for all $M\in2\N_{0}+1$. Denote by $\psi^{\Om_{0}}_{0M}$ and $\psi^{\Om_{M}}_{0M}$ the restrictions of $\psi_{0M}$ to $\Om_{0}$ and $\Om_{M}$, respectively. Then, the strict positivity of $\psi_{0M}$ implies that $\|\psi^{\Om_{0}}_{0M}\|>0$ and $\|\psi^{\Om_{M}}_{0M}\|>0$. By Lemma \ref{app-lem-bracketing}, we have
\begin{equation*}
\begin{split}
\la_{0M}\|\psi_{0M}\|^{2}&=\lan\psi_{0M},P_{0M}\psi_{0M}\ran\\
&=\lan\psi^{\Om_{0}}_{0M},P_{0}\psi^{\Om_{0}}_{0M}\ran+\lan\psi^{\Om_{M}}_{0M},H_{a,\Om_{M}}\psi^{\Om_{M}}_{0M}\ran\\
&>E_{0}\|\psi^{\Om_{0}}_{0M}\|^{2}+E_{0}\|\psi^{\Om_{M}}_{0M}\|^{2}\\
&=E_{0}\|\psi_{0M}\|^{2},
\end{split}
\end{equation*}
which leads to the result.

We now prove the theorem for all large $M$. We assume w.l.o.g that there's some $\la_{0}\in(E_{0},\al)$ such that $\la_{0M}\in(E_{0},\la_{0})$. Since $\psi_{0M}$ satisfies the equation $(-\De+W_{0})\psi_{0M}=\la_{0M}\psi_{0M}$ in $\Om_{0}$ and Mezincescu boundary condition on $\pa\Om_{0}\backslash S$, we conclude that $\psi^{\Om_{0}}_{0M}$ is the unique solution to the problem \eqref{eigen-problem} with $g$ replaced by $\Ga_{\Om_{0}}\psi^{\Om_{0}}_{0M}\in H^{3/2}(S)$. Hence,
\begin{equation}\label{equality-on-S}
T_{\la}(\Ga_{\Om_{0}}\psi^{\Om_{0}}_{0M})=\Ga_{\Om_{0}}(\pa_{d}\psi^{\Om_{0}}_{0M}).
\end{equation}
Using Green's formula, we calculate
\begin{equation*}
\begin{split}
&\int_{\Om_{M}}\ol{\psi_{0M}}(P_{0M}\psi_{0M})\\
&\quad\quad=\int_{\Om_{M}}|\nabla\psi_{0M}|^{2}-\int_{\pa\Om_{M}\backslash S}\ol{\psi}(\textbf{n}_{\Om_{M}}\cdot\nabla\psi_{0M})+\int_{S}\ol{\psi_{0M}}(\pa_{d}\psi_{0M})+\int_{\Om_{M}}W_{M}|\psi_{0M}|^{2}\\
&\quad\quad=\int_{\Om_{M}}|\nabla\psi_{0M}|^{2}+\int_{\pa\Om_{M}}\chi_{\Om_{M}}|\psi_{0M}|^{2}+\int_{\Om_{M}}W_{M}|\psi_{0M}|^{2}\\
&\quad\quad\quad\quad+\int_{S}\ol{\Ga_{\Om_{0}}\psi^{\Om_{0}}_{0M}}T_{\la_{0M}}(\Ga_{\Om_{0}}\psi^{\Om_{0}}_{0M})+\int_{S}(-\chi_{\Om_{M}})|\psi^{\Om_{0}}_{0M}|^{2},
\end{split}
\end{equation*}
where we used the fact $\textbf{n}_{\Om_{M}}\cdot\nabla=-\pa_{d}$ on $S$ in the first equality and \eqref{equality-on-S} in the second equality. Since $\textbf{n}_{\Om_{M}}=-\textbf{n}_{\Om_{0}}$ on $S$, we have $\chi_{\Om_{M}}=-\chi_{\Om_{0}}$ on $S$, and hence, by Lemma \ref{lem-estimate}$\rm(ii)$,
\begin{equation*}
\la_{0M}\|\psi^{\Om_{M}}_{0M}\|^{2}=\int_{\Om_{M}}\ol{\psi_{0M}}(P_{0M}\psi_{0M})\geq Q_{M}(\psi^{\Om_{M}}_{0M},\psi^{\Om_{M}}_{0M})+\ep\|\Ga_{M}\psi^{\Om_{M}}_{0M}\|^{2},
\end{equation*}
where we used the obvious fact that $\Ga_{M}\psi^{\Om_{M}}_{0M}=\Ga_{\Om_{0}}\psi^{\Om_{0}}_{0M}$. We now apply Lemma \ref{lem-estimate-1} to conclude that for all large $M\in2\N_{0}+1$
\begin{equation*}
\la_{0M}\|\psi^{\Om_{M}}_{0M}\|^{2}\geq E_{0}\|\psi^{\Om_{M}}_{0M}\|^{2}+\frac{C}{M^{2}}\|\psi^{\Om_{M}}_{0M}\|^{2}
\end{equation*}
for some $C>0$. The theorem follows since $\|\psi^{\Om_{M}}_{0M}\|>0$. This completes the proof.
\end{proof}

\subsection{Non-Inherited Ground State Energy}\label{subsec-non-inherited-gse}

We prove Theorem \ref{thm-lower-bound-gse-main} in the case $W_{\Om_{0M}}1_{\Om_{M}}=b\sum_{i\in\Z^{d}\cap\Om_{M}}u(\cdot-i)$. Consider operators defined as follows: let $P_{0}$ be the same as in the Subsection \ref{subsec-inherited-gse}; for $M\in2\N_{0}+1$, let $\Om_{M}=(-\frac{1}{2},\frac{1}{2})^{d-1}\times(-\frac{1}{2},\frac{M}{2})$ and set
\begin{equation*}
H_{b,\Om_{M}}=-\De_{\Om_{M}}+W_{M}^{*}=-\De_{\Om_{M}}+V_{0}1_{\Om_{M}}+b\sum_{i\in\Z^{d}\cap\Om_{M}}u(\cdot-i);
\end{equation*}
let $\Om_{0M}=\text{int}(\ol{\Om_{0}\cup\Om_{M}})$ and set
\begin{equation*}
P_{0M}^{*}=-\De_{\Om_{0M}}+W_{0M}^{*},
\end{equation*}
where $W_{0M}^{*}=W_{0}1_{\Om_{0}}+W_{M}^{*}1_{\Om_{M}}$.

Note $H_{b,\Om_{M}}$ is the operator $H_{b}$ restricted to $\Om_{M}$ with Mezincescu boundary condition. Let $E_{M}^{*}$ and $\vp_{M}^{*}$ be the ground state energy and the strictly positive ground state, respectively, of $H_{b,\Om_{M}}$. Since the Mezincescu boundary condition are defined via $\vp_{a}$, $E_{M}^{*}$ is not inherited from $\inf\si(H_{b})$, and, from Lemma \ref{app-lem-bracketing}, we can only conclude that $E_{M}^{*}\geq E_{0}$ for all $m\in2\N_{0}+3$.

Theorem \ref{thm-lower-bound-gse-main} in the case $W_{\Om_{0M}}1_{\Om_{M}}=b\sum_{i\in\Z^{d}\cap\Om_{M}}u(\cdot-i)$ is restated as

\begin{thm}\label{thm-lower-bound-estimate-1}
If $\inf\si(P_{0})>E_{0}$, then there exists some $M$-independent constant $C>0$ such that
\begin{equation*}
\inf\si(P_{0M}^{*})\geq E_{0}+\frac{C}{M^{2}}
\end{equation*}
for all $M\in2\N_{0}+1$.
\end{thm}

To prove Theorem \ref{thm-lower-bound-estimate-1}, we need the following two lemmas refining $E_{M}^{*}\geq E_{0}$ for all $m\in2\N_{0}+3$.

\begin{lem}\label{lem-key}
There holds the alternative: either
\begin{itemize}
\item[\rm(i)] $E_{M}^{*}=E_{0}$ for all $M\in2\N_{0}+3$, or
\end{itemize}
\begin{itemize}
\item[\rm(ii)] there's some $\de>0$ such that $E_{M}^{*}\geq E_{0}+\de$ for all $M\in2\N_{0}+3$.
\end{itemize}
Moreover, if $\rm(i)$ is satisfied, then there exists some constant $\nu>0$ such that
\begin{equation*}
\vp_{M}^{*}|_{\CC_{(0,r)}}=\nu^{r}\vp_{M}^{*}|_{\CC_{(0,0)}}(\cdot-(0,r)),\quad r=0,1\dots,\frac{M-1}{2}
\end{equation*}
for all $M\in2\N_{0}+1$. In particular, if $\nu=1$, then there exist $0<c_{1}<c_{2}\leq1$ such that $\frac{\inf\vp_{M}^{*}}{\sup\vp_{M}^{*}}\in[c_{1},c_{2}]$ for all $M\in2\N_{0}+1$
\end{lem}
\begin{proof}
Clearly, $E_{1}^{*}=E_{\vp_{a}}(b)=E_{0}$. We first claim that either $\rm(1)$ $E_{M}^{*}=E_{0}$ for all $M\in2\N_{0}+3$, or $\rm(2)$ $E_{M}^{*}>E_{0}$ for all $M\in2\N_{0}+3$. For contradiction, suppose $\rm(2)$ fails, that is, there's some $M_{0}\in2\N_{0}+3$ such that $E_{M_{0}}^{*}=E_{0}$, and show $\rm(1)$ holds. 

We first show that $E_{M}^{*}=E_{0}$ for all $M\in2\N_{0}+3$ satisfying $M<M_{0}$. Fix any such an $M$. Let $\Om_{-}=\Om_{M}=(-\frac{1}{2},\frac{1}{2})^{d-1}\times(-\frac{1}{2},\frac{M}{2})$ and $\Om_{+}=(-\frac{1}{2},\frac{1}{2})^{d-1}\times(\frac{M}{2},\frac{M_{0}}{2})$. Set $\vp_{\pm}=\vp_{M_{0}}^{*}|_{\Om_{\pm}}$. By Lemma \ref{app-lem-bracketing}, we find
\begin{equation*}
\lan\vp_{-},H_{b,\Om_{-}}\vp_{-}\ran+\lan\vp_{+},H_{b,\Om_{+}}\vp_{+}\ran=\lan\vp_{M_{0}}^{*},H_{b,\Om_{M_{0}}}\vp_{M_{0}}^{*}\ran=E_{0}\|\vp_{-}\|^{2}+E_{0}\|\vp_{+}\|^{2}.
\end{equation*}
Since $\inf\si(H_{b,\Om_{-}})\geq E_{0}$ and $\inf\si(H_{b,\Om_{+}})\geq E_{0}$, there holds
\begin{equation*}
\lan\vp_{-},H_{b,\Om_{-}}\vp_{-}\ran=E_{0}\|\vp_{-}\|^{2}\quad\text{and}\quad\lan\vp_{+},H_{b,\Om_{+}}\vp_{+}\ran=E_{0}\|\vp_{+}\|^{2}.
\end{equation*}
Variational principle and the uniqueness of ground state then yield that $\vp_{-}$ is the ground state of $H_{b,\Om_{-}}$, which leads to $E_{M}^{*}=E_{0}$ and the claim follows.

We next show that $E_{M_{0}+1}^{*}=E_{0}$. For $i\in\Z^{d}\cap\Om_{M_{0}}$, we set $\vp_{i}=\vp_{M_{0}}^{*}|_{\CC_{i}}$. Similar arguments as above show that $\vp_{i}$ is the ground state of $H_{b,\CC_{i}}$. Since $H_{b,\CC_{i}}$, $i\in\Z^{d}\cap\Om_{M_{0}}$ are unitarily equivalent, $\vp_{i}$, $i\in\Z^{d}\cap\Om_{M_{0}}$ are all same up to translations and multiplication by positive scalars. In particular, there's some constant $\nu>0$ such that
\begin{equation*}
\vp_{M_{0}}^{*}|_{\CC_{(0,0)}}=\nu\vp_{M_{0}}^{*}|_{\CC_{(0,0)}}(\cdot-(0,1)).
\end{equation*}
We now define the continuous function $\tilde{\vp}_{M_{0}}^{*}:\Om_{M_{0}}\ra(0,\infty)$ by setting
\begin{equation}\label{construction-gs}
\begin{split}
&\tilde{\vp}_{M_{0}}^{*}|_{\CC_{(0,0)}}=\vp_{M_{0}}^{*}|_{\CC_{(0,0)}},\\
&\tilde{\vp}_{M_{0}}^{*}|_{\CC_{(0,r)}}=\nu^{r}\tilde{\vp}_{M}^{*}|_{\CC_{(0,0)}}(\cdot-(0,r)),\quad r=1\dots,\frac{M-1}{2}.
\end{split}
\end{equation}
Then, $\tilde{\vp}_{M_{0}}^{*}$ is the ground state of $H_{b,\Om_{M_{0}}}$. By uniqueness, there holds $\vp_{M_{0}}^{*}=\tilde{\vp}_{M_{0}}^{*}$. Therefore, we can easily construct a ground state of $H_{b,\Om_{M_{0}+1}}$ and conclude that $E_{M_{0}+1}^{*}=E_{0}$. 

By induction, $E_{M}^{*}=E_{0}$ for all $M\geq M_{0}+1$. This proves the claim, that is, either $E_{M}^{*}=E_{0}$ for all $M\in2\N_{0}+3$ or $E_{M}^{*}>E_{0}$ for all $M\in2\N_{0}+3$. To finish the proof, we assume $E_{M}^{*}>E_{0}$ for all $M\in2\N_{0}+3$ and show that $E_{M}^{*}\geq E_{0}+\de$ for all $M\in2\N_{0}+3$ for some $\de>0$. Let
\begin{equation*}
\de=\min\{\inf\si(H_{b,\Om_{3}}),\inf\si(H_{b,\Om_{5}})\}-E_{0}>0.
\end{equation*}
Since for $M\in2\N_{0}+3$ any $\Om_{M}$ is the disjoint union of subdomains and each of these subdomains is either the union of 2 adjacent cubes or the union of 3 adjacent cubes, we conclude from the unitary equivalence and Lemma \ref{app-lem-bracketing} that
\begin{equation*}
E_{M}^{*}\geq\min\{\inf\si(H_{b,\Om_{3}}),\inf\si(H_{b,\Om_{5}})\}=E_{0}+\de
\end{equation*}
for all $M\in2\N_{0}+3$. 

The ``moreover'' part follows from \eqref{construction-gs}. This completes the proof.
\end{proof}

If $E_{\vp_{a}}(a)=E_{\vp_{a}}(b)$ and $E_{\vp_{b}}(a)=E_{\vp_{b}}(b)$, we can obtain more information about $\nu$ from the condition $E_{M}^{*}=E_{0}$ for some, hence for all, $M\in2\N_{0}+3$ .

\begin{lem}\label{lem-key-2}
Suppose $E_{M}^{*}=E_{0}$ for all $M\in2\N_{0}+3$. If $E_{\vp_{a}}(a)=E_{\vp_{a}}(b)$ and $E_{\vp_{b}}(a)=E_{\vp_{b}}(b)$ is satisfied, then $\nu=1$.
\end{lem}
\begin{proof}
Set $\SS_{0M}=(-\frac{1}{2},\frac{1}{2})^{d-1}\times(-\frac{M}{2},\frac{M}{2})$. As in \eqref{construction-gs}, we can easily construct a strict positive ground state, denoted by $\vp_{\SS_{0M}}^{*}$, of $H_{b,\SS_{0M}}$ with ground state energy $E_{0}$. 

For $q\in\Z^{d}\cap(-\frac{M}{2},\frac{M}{2})^{d-1}$, we set $\SS_{qM}=(q,0)+\SS_{0M}$. It follows from Lemma \ref{app-lem-bracketing} that 
\begin{equation*}
H_{b,\La_{M}}\geq\bigoplus_{q\in\Z^{d}\cap(-\frac{M}{2},\frac{M}{2})^{d-1}}H_{b,\SS_{qM}}.
\end{equation*}
By unitary equivalence, we find $\inf\si(H_{b,\La_{M}})\geq\inf\si(H_{b,\SS_{0M}})=E_{0}$.

Since $E_{\vp_{b}}(a)=E_{\vp_{b}}(b)$, Theorem \ref{thm-spectral-minimum}$\rm(ii)$ ensures that $E_{0}=E_{\vp_{b}}(b)$. Since $E_{\vp_{b}}(b)$ is inherited from $H_{b}$ by Lemma \ref{lem-Mezin}, we have $E_{0}=\inf\si(H_{b})$. Again, using Lemma \ref{app-lem-bracketing}, there holds $H_{b}\geq\bigoplus_{i\in M\Z^{d}}H_{b,i+\La_{M}}$, which leads to $E_{0}=\inf\si(H_{b})\geq\inf\si(H_{b,\La_{M}})$. Hence, 
\begin{equation}\label{spectrum-bottom}
\inf\si(H_{b,\La_{M}})=E_{0}.
\end{equation}

Using \eqref{spectrum-bottom} and the unitary equivalence of operators $H_{b,\SS_{qM}}$, $q\in\Z^{d}\cap(-\frac{M}{2},\frac{M}{2})^{d-1}$, a similar argument as in the proof of Lemma \ref{lem-key} yields that the ground state of $H_{b,\La_{M}}$ restricted to each $\SS_{qM}$ is the same as $\vp_{\SS_{0M}}^{*}$ up to translations and multiplication by positive scalars. 

Clearly, the above argument holds for any $M\in2\N_{0}+3$. Therefore, we actually obtain a ground state, denoted by $\vp_{b}^{*}$, of $H_{b}$.  The $\Z^{d}$-perodicity of $H_{b}$ then implies the $\Z^{d}$-periodicity of $\vp_{b}^{*}$, which leads to $\nu=1$.
\end{proof}

We now prove Theorem \ref{thm-lower-bound-estimate-1}.

\begin{proof}[Proof of Theorem \ref{thm-lower-bound-estimate-1}]
For all not-very-large $M$, the result follows from the arguments as in the proof of Theorem \ref{thm-lower-bound-estimate}. For all large $M$, using Lemma \ref{lem-key}, we only need to consider two cases. If Lemma \ref{lem-key}$\rm(ii)$ is true, we set $\de^{*}=\min\{\inf\si(P_{0})-E_{0},\de\}>0$ and conclude from Lemma \ref{app-lem-bracketing} that $\inf\si(P_{0M}^{*})\geq\min\{\inf\si(P_{0}),E_{M}^{*}\}\geq E_{0}+\de^{*}$, which leads to the result.

We now suppose that Lemma \ref{lem-key}$\rm(i)$ is satisfied. As \eqref{estimate-2}, we can use the ground state transform to find
\begin{equation*}
\bigg\|\nabla\bigg(\frac{\psi}{\vp_{M}^{*}}\bigg)\bigg\|^{2}\leq\frac{1}{(\inf\vp_{M}^{*})^{2}}\big[Q^{*}_{M}(\psi,\psi)-E_{0}\|\psi\|^{2}\big],\quad\psi\in H^{1}(\Om_{M}),
\end{equation*}
where $Q_{M}^{*}$ is the quadratic form of $H_{b,\Om_{M}}$. \eqref{estimate-1} and the above estimate ensures a similar estimate as in Lemma \ref{lem-estimate-1}, that is,
\begin{equation*}
\frac{4}{M}\|\Ga_{M}\psi\|^{2}+Q^{*}_{M}(\psi,\psi)-E_{0}\|\psi\|^{2}\geq\bigg(\frac{\inf\vp_{M}^{*}}{\sup\vp_{M}^{*}}\bigg)^{2}\frac{4}{M(M+1)}\|\psi\|^{2},\quad\psi\in H^{1}(\Om_{M}).
\end{equation*}
Note $\frac{\inf\vp_{M}^{*}}{\sup\vp_{M}^{*}}\in[c_{1},c_{2}]$ for all $M\in2\N_{0}+1$ by Lemma \ref{lem-key}, Lemma\ref{lem-key-2} and assumption $\rm(H6)$, which ensures the usefulness of the above estimate. The remaining proof follows in the same way as that of the proof of Theorem \ref{thm-lower-bound-estimate}.
\end{proof}

We end this section by making the following remark.

\begin{rem}\label{rem-lower-bound-gse}
\begin{itemize}
\item[\rm(i)] As mentioned at the beginning of this section, we will need a counterpart of Theorem \ref{thm-lower-bound-gse-main} .  Let
\begin{equation}\label{domain-2}
\begin{split}
\Om_{0}&=\bigg(-\frac{1}{2},\frac{1}{2}\bigg)^{d-1}\times\bigg(\frac{1}{2},\frac{m}{2}\bigg),\quad m\in2\N_{0}+3,\\
\Om_{M}&=\bigg(-\frac{1}{2},\frac{1}{2}\bigg)^{d-1}\times\bigg(-\frac{M}{2},\frac{1}{2}\bigg),\quad M\in2\N_{0}+1,\\
\Om_{0M}&=\text{\rm int}(\ol{\Om_{0}\cup\Om_{M}})
\end{split}
\end{equation}
and consider the operator
\begin{equation*}
-\De_{\Om_{0M}}+V_{0}1_{\Om_{0M}}+W_{\Om_{0M}}\quad\text{on}\quad\Om_{0M},
\end{equation*}
where the potential $W_{\Om_{0M}}$ is defined as follows: $W_{\Om_{0M}}1_{\Om_{0}}$ is of the form $\sum_{i\in\Z^{d}\cap\Om_{0}}\om_{i}u(\cdot-i)$ and $W_{\Om_{0M}}1_{\Om_{M}}=a\sum_{i\in\Z^{d}\cap\Om_{M}}u(\cdot-i)$ or $b\sum_{i\in\Z^{d}\cap\Om_{M}}u(\cdot-i)$.

\begin{thm}\label{thm-lower-bound-gse-main-1}
If $\inf\si(-\De_{\Om_{0}}+V_{0}1_{\Om_{0}}+W_{\Om_{0M}}1_{\Om_{0}})>E_{0}$, then there's some $M$-independent $C>0$ such that
\begin{equation*}
\inf\si(-\De_{\Om_{0M}}+V_{0}1_{\Om_{0M}}+W_{\Om_{0M}})\geq E_{0}+\frac{C}{M^{2}}
\end{equation*}
for all $M\in2\N_{0}+1$.
\end{thm}

\item[\rm(ii)] Theorem \ref{thm-lower-bound-gse-main} and Theorem \ref{thm-lower-bound-gse-main-1} will be used in the next section only through
\begin{itemize}
\item[$\bullet$] $m=3$ and $W_{\Om_{0M}}1_{\Om_{0}}=tu(\cdot-i)$ with $\Z^{d}\cap\Om_{0}=\{i\}$ for some suitable $t\in(a,b)$;
\item[$\bullet$] $m=5$ and $W_{\Om_{0M}}1_{\Om_{0}}=t_{1}u(\cdot-i_{1})+t_{2}u(\cdot-i_{2})$ with $\Z^{d}\cap\Om_{0}=\{i_{1},i_{2}\}$, $t_{1},t_{2}\in\{a,b\}$ and $t_{1}\neq t_{2}$.
\end{itemize}
Under certain assumptions, the condition $\inf\si(-\De_{\Om_{0}}+V_{0}1_{\Om_{0}}+W_{\Om_{0M}}1_{\Om_{0}})>E_{0}$ in Theorem \ref{thm-lower-bound-gse-main} and Theorem \ref{thm-lower-bound-gse-main-1} is satisfied in both cases.
\end{itemize}
\end{rem}

\section{Lifshitz Tails: Non-Optimal Upper Bound}\label{sec-LT-non-optimal-upper-bound}

We prove Theorem \ref{thm-Lifshitz-tail-non-critical} and Theorem \ref{thm-Lifshitz-tail-non} in this section. To fix the ideal, we focus on the case $E_{\vp_{a}}(a)=E_{\vp_{a}}(b)$ and $E_{\vp_{b}}(a)\leq E_{\vp_{b}}(b)$. Therefore, all the Mezincescu boundary conditions in this section are defined using $\vp_{a}$. Also, to simplify statements, we always assume
\begin{itemize}
\item[$\bullet$] $\rm (H1)$, $\rm(H2)$, $\rm(H3)$, $\rm(H4)$, $\rm(H5)$, $\rm(H6)$, $E_{\vp_{a}}(a)=E_{\vp_{a}}(b)$ and $E_{\vp_{b}}(a)\leq E_{\vp_{b}}(b)$.
\end{itemize}
Due to technical reasons, we treat Bernoulli models and non-Bernoulli models separately. Theorem \ref{thm-Lifshitz-tail-non} is restated in Theorem \ref{thm-Lifshitz-tail-non-optimal} for non-Bernoulli models and in Theorem \ref{thm-Lifshitz-tail-non-optimal-B} for Bernoulli models.

\subsection{Non-Bernoulli Models}\label{subsec-non-B-model}

We treat non-Bernoulli models, that is, the i.i.d random variables $\{\om_{i}\}_{i\in\Z^{d}}$ are not Bernoulli distributed, so we can find some $\ep>0$ such that
\begin{equation}\label{the-prob}
\mu=\P\big\{\om\in\Om\big|\om_{*}\in[a,a+\ep)\cup(b-\ep,b]\big\}\in(0,1),
\end{equation}
where $\om_{*}$ the universal representation of $\{\om_{i}\}_{i\in\Z^{d}}$. We fix such an $\ep$. 

Theorem \ref{thm-Lifshitz-tail-non} in this case is restated as

\begin{thm}\label{thm-Lifshitz-tail-non-optimal}
If the i.i.d random variables $\{\om_{i}\}_{i\in\Z^{d}}$ are not Bernoulli distributed, then
\begin{equation*}
\limsup_{E\da E_{0}}\frac{\ln|\ln N(E)|}{\ln(E-E_{0})}\leq-\frac{1}{2}.
\end{equation*}
\end{thm}

For $L\in2\N_{0}+1$, set $E_{L}(\om)=\inf\si(H_{\om,\La_{L}})$. Since $H_{\om,\La_{L}}$ depends only on $\{\om_{i}\}_{i\in\Z^{d}\cap\La_{L}}$, so does $E_{L}(\om)$. It's not hard to verify that the map $\om\mapsto E_{L}(\om):[a,b]^{\Z^{d}\cap\La_{L}}\ra\R$ is real analytic and concave. The following lemma is the key to the proof of the above theorem.

\begin{lem}\label{lem-estimate-key-4}
There exists some $C>0$ such that for all large $L\in2\N_{0}+1$ there holds
\begin{equation}\label{estimate-key-1}
E_{L}(\om)\geq E_{0}+\frac{C}{L^{2}}
\end{equation}
for all $\om\in[a,b]^{\Z^{d}\cap\La_{L}}$ satisfying the property: for any $q\in\Z^{d-1}\cap(-\frac{L}{2},\frac{L}{2})^{d-1}$ there exist $r_{1},r_{2}\in\Z\cap(-\frac{L}{2},\frac{L}{2})$ such that $|r_{1}-r_{2}|=1$ (that is, $r_{1}$ and $r_{2}$ are adjacent) and both $\om_{(q,r_{1})}$ and $\om_{(q,r_{2})}$ belong to $[a+\ep,b-\ep]$.
\end{lem}

The proof of Lemma \ref{lem-estimate-key-4} is technical. Let's postpone it to the proof of Theorem \ref{thm-Lifshitz-tail-non-optimal}.

\begin{proof}[Proof of Theorem \ref{thm-Lifshitz-tail-non-optimal}]
It suffices to give a proper estimate for $\P\{\om\in\Om|E_{L}(\om)\leq E\}$ for large $L\in2\N_{0}+1$. Let $E>E_{0}$ and set $L=c(E-E_{0})^{-1/2}$ for some $c>0$ with $c^{2}<C$, where $C>0$ is the same as in \eqref{estimate-key-1}. Assume that $E$ is close to $E_{0}$ so that $L$ is large. If $\om\in[a,b]^{\Z^{d}\cap\La_{L}}$ is as in Lemma \ref{lem-estimate-key-4}, we deduce from \eqref{estimate-key-1} that $E_{L}(\om)>E$. Therefore,
\begin{equation*}
\begin{split}
&\P\big\{\om\in\Om\big|E_{L}(\om)\leq E\big\}\\
&\quad\quad\leq\P\bigg\{\om\in[a,b]^{\Z^{d}\cap\La_{L}}\bigg|
\begin{aligned}
&\exists q\in\Z^{d-1}\cap(-L/2,L/2)^{d-1}\,\,\text{s.t.}\,\,\text{no adjacent}\\ &r_{1},r_{2}\in\Z\cap(-L/2,L/2)\,\,\text{satisfies}\,\,\om_{(q,r_{1})},\om_{(q,r_{2})}\in[a+\ep,b-\ep]
\end{aligned}\bigg\}\\
&\quad\quad\leq\sum_{q\in\Z^{d-1}\cap(-\frac{L}{2},\frac{L}{2})^{d-1}}\P\bigg\{\om\in[a,b]^{\Z^{d}\cap\SS_{qL}}\bigg|
\begin{aligned}
&\text{no adjacent}\,\,r_{1},r_{2}\in\Z\cap(-L/2,L/2)\\
&\text{satisfies}\,\,\om_{(q,r_{1})},\om_{(q,r_{2})}\in[a+\ep,b-\ep]
\end{aligned}\bigg\}\\
&\quad\quad\triangleq\sum_{q\in\Z^{d-1}\cap(-\frac{L}{2},\frac{L}{2})^{d-1}}\P\big\{\Om_{q}\big\}\\
&\quad\quad=L^{d-1}\P\big\{\Om_{q_{0}}\big\}
\end{split}
\end{equation*}
for any $q_{0}\in\Z^{d-1}\cap(-\frac{L}{2},\frac{L}{2})^{d-1}$.

To estimate the probability $\P\{\Om_{q_{0}}\}$, we note that the event $\Om_{q_{0}}$ can be written as
\begin{equation*}
\begin{split}
\Om_{q_{0}}=\Bigg\{\om\in[a,b]^{\Z^{d}\cap\SS_{q_{0}L}}\Bigg|
\begin{aligned}
&\text{any adjacent}\,\,r_{1},r_{2}\in\Z\cap(-L/2,L/2)\,\,\text{satisfies either}\\
&\om_{(q_{0},r_{1})}\in[a+\ep,b-\ep]\,\,\text{and}\,\,\om_{(q_{0},r_{2})}\in[a,a+\ep)\cup(b-\ep,b]\\
&\text{or}\,\,\om_{(q_{0},r_{1})}\in[a,a+\ep)\cup(b-\ep,b]\,\,\text{and}\,\,\om_{(q_{0},r_{2})}\in[a+\ep,b-\ep]\\
\end{aligned}\Bigg\}.
\end{split}
\end{equation*}
Let $N=\max\{n\in\Z|2n\leq\frac{L}{2}\}$, the largest integer satisfying $2N\leq\frac{L}{2}$. For $n=1,2,\dots,N$, set $I_{n}=\big\{2n-1,2n\big\}$ and for $n=-1,-2,\dots,-N$, set $I_{n}=\big\{2n,2n+1\big\}$.
That is, we decompose the sets $\{1,\dots,2N\}$ and $\{-2N,\dots,-1\}$ into disjoint sets such that each such set consists of two adjacent integers. Then, for any $n\in\{-N,\dots,N\}\bs\{0\}$, we can simply write
\begin{equation*}
I_{n}=\big\{r_{n1},r_{n2}\big\}
\end{equation*}
with $r_{n2}-r_{n1}=1$. Moreover, for any $m,n\in\{-N,\dots,N\}\bs\{0\}$ with $m\neq n$,
\begin{equation*}
\om_{(q_{0},I_{m})}=\{\om_{(q_{0},r_{m1})},\om_{(q_{0},r_{m2})}\}\quad\text{and}\quad\om_{(q_{0},I_{n})}=\{\om_{(q_{0},r_{n1})},\om_{(q_{0},r_{n2})}\}
\end{equation*}
are independent. It follows that
\begin{equation*}
\begin{split}
\P\big\{\Om_{q_{0}}\big\}&\leq\P\Bigg\{\om\in[a,b]^{\Z^{d}\cap\SS_{q_{0}L}}\Bigg|
\begin{aligned}
&\forall n\in\{-N,\dots,N\}\bs\{0\}\,\,\text{there holds either}\\
&\om_{(q_{0},r_{n1})}\in[a+\ep,b-\ep]\,\,\text{and}\,\,\om_{(q_{0},r_{n2})}\in[a,a+\ep)\cup(b-\ep,b]\\
&\text{or}\,\,\om_{(q_{0},r_{n1})}\in[a,a+\ep)\cup(b-\ep,b]\,\,\text{and}\,\,\om_{(q_{0},r_{n2})}\in[a+\ep,b-\ep]\\
\end{aligned}\Bigg\}\\
&=\prod_{n\in\{-N,\dots,N\}\bs\{0\}}\P\big\{\Om_{q_{0}}(n)\big\}\\
&=\Big(\P\big\{\Om_{q_{0}}(n_{0})\big\}\Big)^{2N}
\end{split}
\end{equation*}
for any $n_{0}\in\{-N,\dots,N\}\bs\{0\}$, where
\begin{equation*}
\Om_{q_{0}}(n)=\bigg\{\om\in[a,b]^{(q_{0},I_{n})}\bigg|
\begin{aligned}
&\text{either}\,\,\om_{(q_{0},r_{n1})}\in[a+\ep,b-\ep]\,\,\text{and}\,\,\om_{(q_{0},r_{n2})}\in[a,a+\ep)\cup(b-\ep,b]\\
&\text{or}\,\,\om_{(q_{0},r_{n1})}\in[a,a+\ep)\cup(b-\ep,b]\,\,\text{and}\,\,\om_{(q_{0},r_{n2})}\in[a+\ep,b-\ep]\\
\end{aligned}\bigg\}
\end{equation*}
for $n\in\{-N,\dots,N\}\bs\{0\}$. It's easy to see $\P\big\{\Om_{q_{0}}(n_{0})\big\}=2\mu(1-\mu)\leq\frac{1}{2}$,
which leads to
\begin{equation*}
\P\big\{\Om_{q_{0}}\big\}\leq\bigg(\frac{1}{2}\bigg)^{2N}\leq2^{\frac{3}{2}}\bigg(\frac{1}{2}\bigg)^{\frac{L}{2}},
\end{equation*}
where we used $4N\geq L-3$. Consequently, recalling $L=c(E-E_{0})^{-1/2}$, we find
\begin{equation}\label{main-estimate}
\P\big\{\om\in\Om\big|E_{L}(\om)\leq E\big\}\leq L^{d-1}2^{\frac{3}{2}}\bigg(\frac{1}{2}\bigg)^{\frac{L}{2}}=c^{d-1}(E-E_{0})^{-(d-1)/2}2^{\frac{3}{2}}\bigg(\frac{1}{2}\bigg)^{\frac{c}{2}(E-E_{0})^{-1/2}},
\end{equation}
for all $E>E_{0}$ with $E-E_{0}$ small, which leads to the result.
\end{proof}

We point out that a more direct approach to an estimate similar to \eqref{main-estimate} is given in Remark \ref{rem-a-prob-estimate} below when $\mu$ is in a neighborhood of $\frac{1}{2}$.

We now proceed to prove Lemma \ref{lem-estimate-key-4}. By Lemma \ref{lem-ground-state}$\rm(i)$, we have $E_{\vp_{a}}(a)<E_{\vp_{a}}(a+\ep)$ and $E_{\vp_{a}}(b-\ep)>E_{\vp_{a}}(b)$. In particular,
\begin{equation}\label{a-condition}
\min\big\{E_{\vp_{a}}(a+\ep),E_{\vp_{a}}(b-\ep)\big\}>E_{0}.
\end{equation}

\begin{lem}\label{lem-estimate-key-1}
Let $r\in\Z\cap(-\frac{L}{2},\frac{L}{2})$ and set
$\SS=(-\frac{1}{2},\frac{1}{2})^{d-1}\times(r-\frac{1}{2},\frac{L}{2})$.
Then there exists some $C>0$ such that for all large $L\in2\N_{0}+1$ there holds
\begin{equation*}
\inf\si(H_{\om,\SS})\geq E_{0}+\frac{C}{\big(\frac{L+1}{2}-r\big)^{2}}
\end{equation*}
for all $\om\in\{a,b,a+\ep,b-\ep\}^{\Z^{d}\cap\SS}$ satisfying $\om_{(0,r)}\in\{a+\ep,b-\ep\}$.
\end{lem}
\begin{proof}
Let $\om\in\{a,b,a+\ep,b-\ep\}^{\Z^{d}\cap\SS}$ satisfy $\om_{(0,r)}\in\{a+\ep,b-\ep\}$. We claim that there exist $K\in\N$ and  subsegments $\SS_{1},\dots,\SS_{K}$ satisfying following conditions:
\begin{itemize}
\item[\rm(i)] for each $k\in\{1,\dots,K\}$, there are $l_{k},m_{k}\in\Z\cap(-\frac{L}{2},\frac{L}{2})$ with $l_{k}\leq m_{k}$ such that $\SS_{k}=(-\frac{1}{2},\frac{1}{2})^{d-1}\times(l_{k}-\frac{1}{2},m_{k}+\frac{1}{2})$;

\item[\rm(ii)] $\SS_{1},\dots,\SS_{K}$ are pairwise disjoint and $\ol{\SS}=\cup_{k=1}^{K}\ol{\SS}_{k}$;

\item[\rm(iii)] for each $k\in\{1,\dots,K\}$, $\{\om_{(0,n)},n=l_{k},l_{k}+1,\dots,m_{k}\}$ satisfies one of the following two conditions:
\begin{itemize}
\item[$\bullet$] $\om_{(0,l_{k})}\in\{a+\ep,b-\ep\}$ and $\om_{(0,n)}=a$ for all $n=l_{k}+1,\dots,m_{k}$ or $\om_{(0,n)}=b$ for all $n=l_{k}+1,\dots,m_{k}$;

\item[$\bullet$] $\om_{(0,l_{k})},\om_{(0,l_{k}+1)}\in\{a,b\}$ with $\om_{(0,l_{k})}\neq\om_{(0,l_{k}+1)}$ and $\om_{(0,n)}=a$ for all $n=l_{k}+2,\dots,m_{k}$ or $\om_{(0,n)}=b$ for all $n=l_{k}+2,\dots,m_{k}$.
\end{itemize}
\end{itemize}
Indeed, the above claim is a consequence of the following iteration steps:

\noindent\textbf{Step 1.} By assumption $\om_{(0,r)}\in\{a+\ep,b-\ep\}$. Let $r_{1}\in(r-\frac{1}{2},\frac{L}{2})$ be such that $\om_{0,n}\in\{a+\ep,b-\ep\}$ for all $n=r,r+1,\dots,r_{1}$ and $\om_{(0,r_{1}+1)}\in\{a,b\}$. For each $n=r,r+1,\dots,r_{1}-1$, we set $(-\frac{1}{2},\frac{1}{2})^{d-1}\times(n-\frac{1}{2},n+\frac{1}{2})$ to be a subsegment.

\noindent\textbf{Step 2.} Let $r_{2}\geq r_{1}+1$ be such that $\om_{(0,r_{2})}=\om_{(0,r_{2}-1)}=\cdots=\om_{(0,r_{1}+1)}$ and $\om_{0,r_{2}+1}\neq\om_{(0,r_{2})}$.
\begin{itemize}
\item[\rm(i)] If $\om_{0,r_{2}+1}\in\{a+\ep,b-\ep\}$, we set $(-\frac{1}{2},\frac{1}{2})^{d-1}\times(r_{1}-\frac{1}{2},r_{2}+\frac{1}{2})$ to be a subsegment and proceed in the same way as in Step 1 with the initial $\om_{0,r_{2}+1}\in\{a+\ep,b-\ep\}$.

\item[\rm(ii)] If $\om_{0,r_{2}+1}\in\{a,b\}$, we set $(-\frac{1}{2},\frac{1}{2})^{d-1}\times(r_{1}-\frac{1}{2},r_{2}-\frac{1}{2})$ to be a subsegment and check $\om_{(0,r_{2}+2)}$.

    \begin{itemize}
    \item[$\bullet$] If $\om_{(0,r_{2}+2)}\in\{a+\ep,b-\ep\}$, we set $(-\frac{1}{2},\frac{1}{2})^{d-1}\times(r_{2}-\frac{1}{2},r_{2}+1+\frac{1}{2})$ to be a subsegment and proceed in the same way as in Step 1 with the initial $\om_{(0,r_{2}+2)}\in\{a+\ep,b-\ep\}$.
    \item[$\bullet$] If $\om_{(0,r_{2}+2)}\in\{a,b\}$, let $r_{3}\geq r_{2}+2$ be such that $\om_{(0,r_{3})}=\om_{(0,r_{3}-1)}=\cdots=\om_{(0,r_{2}+2)}$ and $\om_{0,r_{3}+1}\neq\om_{(0,r_{3})}$.

        \begin{itemize}
            \item[$\bullet$] If $\om_{0,r_{3}+1}\in\{a+\ep,b-\ep\}$, we set $(-\frac{1}{2},\frac{1}{2})^{d-1}\times(r_{2}-\frac{1}{2},r_{3}+\frac{1}{2})$ to be a subsegment and proceed in the same way as in Step 1 with the initial $\om_{0,r_{3}+1}\in\{a+\ep,b-\ep\}$.
            \item[$\bullet$] If $\om_{0,r_{3}+1}\in\{a,b\}$, we set $(-\frac{1}{2},\frac{1}{2})^{d-1}\times(r_{2}-\frac{1}{2},r_{3}-\frac{1}{2})$ to be a subsegment and check $\om_{(0,r_{3}+2)}$. Then, we are in the situation similar to Step 2$\rm(ii)$ and so we can keep iterating.
        \end{itemize}
    \end{itemize}
\end{itemize}

Clearly, any subsegment generated in the above iteration procedure is of the form $(l-\frac{1}{2},m+\frac{1}{2})$ for $l,m\in\Z\cap(r-\frac{1}{2},\frac{L}{2})$ with $l\leq m$ and satisfies one of the following two conditions:
\begin{itemize}
\item[$\bullet$] $\om_{(0,l)}\in\{a+\ep,b-\ep\}$ and $\om_{(0,n)}=a$ for all $n=l+1,\dots,m$ or $\om_{(0,n)}=b$ for all $n=l+1,\dots,m$;

\item[$\bullet$] $\om_{(0,l)},\om_{(0,l+1)}\in\{a,b\}$ with $\om_{(0,l)}\neq\om_{(0,l+1)}$ and $\om_{(0,n)}=a$ for all $n=l+2,\dots,m$ or $\om_{(0,n)}=b$ for all $n=l+2,\dots,m$.
\end{itemize}
Hence, the claim follows.

Now, by Lemma \ref{app-lem-bracketing}, we have $H_{\om,\SS}\geq\bigoplus_{k=1}^{K}H_{\om,\SS_{k}}$, which yields
\begin{equation*}
\inf\si(H_{\om,\SS})\geq\min_{k=1,\dots,K}\inf\si(H_{\om,\SS_{k}}).
\end{equation*}
For each $k\in\{1,\dots,K\}$, we can apply Theorem \ref{thm-lower-bound-gse-main} with \eqref{a-condition} and $\rm(H5)$ to conclude that there's some constant $C>0$ independent of the length of $\SS_{k}$ and $k$ (thus, independent of $\om$) such that
\begin{equation*}
\inf\si(H_{\om,\SS_{k}})\geq E_{0}+\frac{C}{(m_{k}-l_{k}+1)^{2}},
\end{equation*}
which leads to the result of the lemma.
\end{proof}

The following result is a counterpart of Lemma \ref{lem-estimate-key-1}.
\begin{lem}\label{lem-estimate-key-2}
Let $r\in\Z\cap(-\frac{L}{2},\frac{L}{2})$ and set
$\SS=(-\frac{1}{2},\frac{1}{2})^{d-1}\times(-\frac{L}{2},r+\frac{1}{2})$.
Then there exists some $C>0$ such that for all large $L\in2\N_{0}+1$ there holds
\begin{equation*}
\inf\si(H_{\om,\SS})\geq E_{0}+\frac{C}{\big(\frac{L+1}{2}+r\big)^{2}}
\end{equation*}
for all $\om\in\{a,b,a+\ep,b-\ep\}^{\Z^{d}\cap\SS}$ satisfying $\om_{(0,r)}\in\{a+\ep,b-\ep\}$.
\end{lem}

With the help of Lemma \ref{lem-estimate-key-1} and Lemma \ref{lem-estimate-key-2}, we are able to prove the following result, which is the key to Lemma \ref{lem-estimate-key-4}.

\begin{lem}\label{lem-estimate-key-3}
There exists some $C>0$ such that for all large $L\in2\N_{0}+1$ there holds
\begin{equation}\label{estimate-key}
E_{L}(\om)\geq E_{0}+\frac{C}{L^{2}}
\end{equation}
for all $\om\in\{a,b,a+\ep,b-\ep\}^{\Z^{d}\cap\La_{L}}$ satisfying the property: for any $q\in\Z^{d-1}\cap(-\frac{L}{2},\frac{L}{2})^{d-1}$ there exist $r_{1},r_{2}\in\Z\cap(-\frac{L}{2},\frac{L}{2})$ such that $|r_{1}-r_{2}|=1$ (that is, $r_{1}$ and $r_{2}$ are adjacent) and both $\om_{(q,r_{1})}$ and $\om_{(q,r_{2})}$ belong to $\{a+\ep,b-\ep\}$.
\end{lem}
\begin{proof}
For large $L\in2\N_{0}+1$, let $\SS_{0L}=(-\frac{1}{2},\frac{1}{2})^{d-1}\times(-\frac{L}{2},\frac{L}{2})$ and set $\SS_{qL}=(q,0)+\SS_{0L}$ for $q\in\Z^{d-1}$. By Lemma \ref{app-lem-bracketing}, we find
\begin{equation*}
H_{\om,\La_{L}}\geq\bigoplus_{q\in\Z^{d-1}\cap(-\frac{L}{2},\frac{L}{2})^{d-1}}H_{\om,\SS_{qL}},
\end{equation*}
which leads to
\begin{equation}\label{estimate-eig}
E_{L}(\om)\geq\min_{q\in\Z^{d-1}\cap(-\frac{L}{2},\frac{L}{2})^{d-1}}\inf\si(H_{\om,\SS_{qL}}).
\end{equation}

Now, let $\om\in\{a,b,a+\ep,b-\ep\}^{\Z^{d}\cap\La_{L}}$ be as in the  lemma. Fix any $q\in\Z^{d-1}\cap(-\frac{L}{2},\frac{L}{2})^{d-1}$ and consider the operator $H_{\om,\SS_{qL}}$. Let $r_{1},r_{2}\in\Z\cap(-\frac{L}{2},\frac{L}{2})$ be such that $|r_{1}-r_{2}|=1$ and $\om_{(q,r_{1})},\om_{(q,r_{2})}\in\{a+\ep,b-\ep\}$. We may assume w.l.o.g that $r_{1}<r_{2}$. Chopping $\SS_{qL}$ into
\begin{equation*}
\SS_{qL1}=\bigg(-\frac{1}{2},\frac{1}{2}\bigg)^{d-1}\times\bigg(-\frac{L}{2},r_{1}+\frac{1}{2}\bigg),\quad\SS_{qL2}=\bigg(-\frac{1}{2},\frac{1}{2}\bigg)^{d-1}\times\bigg(r_{2}-\frac{1}{2},\frac{L}{2}\bigg)
\end{equation*}
and using Lemma \ref{app-lem-bracketing}, we arrive at
\begin{equation*}
\inf\si(H_{\om,\SS_{qL}})\geq\min\big\{\inf\si(H_{\om,\SS_{qL1}}),\inf\si(H_{\om,\SS_{qL2}})\big\}.
\end{equation*}
Then, we can apply Lemma \ref{lem-estimate-key-1} to $H_{\om,\SS_{qL2}}$ and Lemma \ref{lem-estimate-key-2} to $H_{\om,\SS_{qL1}}$ to conclude that
\begin{equation*}
\inf\si(H_{\om,\SS_{qL}})\geq E_{0}+\min\bigg\{\frac{C_{1}}{(\frac{L+1}{2}+r_{1})^{2}},\frac{C_{2}}{(\frac{L+1}{2}-r_{2})^{2}}\bigg\}\geq E_{0}+\frac{C}{L^{2}}
\end{equation*}
for some $C>0$ independent of $q$ and $L$. The result of the lemma then follows from \eqref{estimate-eig}.
\end{proof}

Finally, we prove Lemma \ref{lem-estimate-key-4}.

\begin{proof}[Proof of Lemma \ref{lem-estimate-key-4}]
Using the concavity of $\om\mapsto E_{L}(\om):[a,b]^{\Z^{d}\cap\La_{L}}\ra\R$ and Lemma \ref{lem-estimate-key-3}, the lemma follows from the two steps argument as in the proof of \cite[Lemma 5.3]{KN10}. We here sketch it for completeness.

Step 1. We first claim that \eqref{estimate-key-1} holds for all $\om\in[a,b]^{\Z^{d}\cap\La_{L}}$ satisfying the property: for any $q\in\Z^{d-1}\cap(-\frac{L}{2},\frac{L}{2})^{d-1}$ there exists $r_{1},r_{2}\in\Z\cap(-\frac{L}{2},\frac{L}{2})$ such that $|r_{1}-r_{2}|=1$ and both $\om_{(q,r_{1})}$ and $\om_{(q,r_{2})}$ belong to $[a+\ep,b-\ep]$, and if $\om_{(q,r)}\notin[a+\ep,b-\ep]$, then $\om_{(q,r)}\in\{a,b\}$. To show this, let $\om$ be as above. Set $\Ga(\om)=\{i\in\Z^{d}\cap\La_{L}|\om_{i}\in[a+\ep,b-\ep]\}$ and $K(\om)=\{a+\ep,b-\ep\}^{\Ga(\om)}$. Then, there are nonnegative coefficients $\{\mu_{\eta}\}_{\eta\in K(\om)}$ satisfying $\sum_{\eta\in K(\om)}\mu_{\eta}=1$ such that $\{\om_{i}\}_{i\in\Ga(\om)}=\sum_{\eta\in K(\om)}\mu_{\eta}\eta$. Then, by setting $\tilde{\eta}_{i}=\eta_{i}$ if $i\in\Ga(\om)$ and $\tilde{\eta}_{i}=\om_{i}$ if $i\notin\Ga(\om)$, we find $\om=\sum_{\eta\in K(\om)}\mu_{\eta}\tilde{\eta}$. Clearly, for each $\eta\in K(\om)$, $\tilde{\eta}$ is as in the statement of Lemma \ref{lem-estimate-key-3}. By the concavity of $\om\mapsto E_{L}(\om):[a,b]^{\Z^{d}\cap\La_{L}}\ra\R$ and Lemma \ref{lem-estimate-key-3}, we find \eqref{estimate-key-1}.

Step 2. Let $\om$ be as in the statement of Lemma \ref{lem-estimate-key-4}. Set $L(\om)=\{a,b\}^{(\Z^{d}\cap\La_{L})\bs\Ga(\om)}$. Then, there are nonnegative coefficients $\{\mu_{\eta}\}_{\eta\in L(\om)}$ satisfying $\sum_{\eta\in L(\om)}\mu_{\eta}=1$ such that $\{\om_{i}\}_{i\in(\Z^{d}\cap\La_{L})\bs\Ga(\om)}=\sum_{\eta\in L(\om)}\mu_{\eta}\eta$. Then, by setting $\tilde{\eta}_{i}=\eta_{i}$ if $i\notin\Ga(\om)$ and $\tilde{\eta}_{i}=\om_{i}$ if $i\in\Ga(\om)$, we find $\om=\sum_{\eta\in L(\om)}\mu_{\eta}\tilde{\eta}$. Clearly, for each $\eta\notin K(\om)$, $\tilde{\eta}$ is as in Step 1. By the concavity of $\om\mapsto E_{L}(\om):[a,b]^{\Z^{d}\cap\La_{L}}\ra\R$ and Step 1, we find \eqref{estimate-key-1}.
\end{proof}

We end this subsection by

\begin{rem}\label{rem-a-prob-estimate}
\begin{itemize}
\item[\rm(i)] We provide a more direct approach to an estimate similar to \eqref{main-estimate} when $\mu\in(1-\frac{1}{\rho},\frac{1}{\rho})$, where $\mu$ is given in \eqref{a-condition} and $\rho=\frac{1+\sqrt{5}}{2}$ is the golden ratio. Recall $L\in2\N_{0}+1$ is sufficiently large and $\P\big\{\om\in\Om\big|E_{L}(\om)\leq E\big\}\leq L^{d-1}\P\big\{\Om_{q}\big\}$
for any $q\in\Z^{d-1}\cap(-\frac{L}{2},\frac{L}{2})^{d-1}$, where
\begin{equation*}
\P\big\{\Om_{q}\big\}=\bigg\{\om\in[a,b]^{\Z^{d}\cap\SS_{qL}}\bigg|
\begin{aligned}
&\text{no adjacent}\,\,r_{1},r_{2}\in\Z\cap(-L/2,L/2)\\
&\text{satisfies}\,\,\om_{(q,r_{1})},\om_{(q,r_{2})}\in[a+\ep,b-\ep]
\end{aligned}\bigg\}.
\end{equation*}
Since the interval $(-\frac{L}{2},\frac{L}{2})$ contains exact $L$ integers, for any $\om\in\Om_{q}$ there are at most $\frac{L+1}{2}$ integers $r_{1},\dots,r_{\frac{L+1}{2}}$ in $(-\frac{L}{2},\frac{L}{2})$ such that $\om_{(q,r_{n})}\in[a+\ep,b-\ep]$ for all $n=1,\dots,\frac{L+1}{2}$ and no two of $\{r_{1},\dots,r_{\frac{L+1}{2}}\}$ are adjacent. Moreover, for any $\om\in\Om_{q}$ there are $\begin{pmatrix}L-N+1\\N\end{pmatrix}$ ways to choose $N$ integers $r_{1},\dots,r_{N}$ from $\Z\cap(-\frac{L}{2},\frac{L}{2})$ such that $\om_{q,r_{n}}\in[a+\ep,b-\ep]$ for all $n=1,\dots,N$ and no two of $\{r_{1},\dots,r_{N}\}$ are adjacent (this can be verified by induction on $N$). Therefore,
\begin{equation*}
\P\big\{\Om_{q}\big\}\leq\sum_{N=0}^{\frac{L+1}{2}}\begin{pmatrix}L-N+1\\N\end{pmatrix}(1-\mu)^{N}\mu^{L-N}.
\end{equation*}
Clearly,
\begin{equation*}
\sum_{N=0}^{\frac{L+1}{2}}\begin{pmatrix}L-N+1\\N\end{pmatrix}=F_{L+2},
\end{equation*}
where $\{F_{n}\}_{n\in\N}$ is the Fibonacci sequence. It is well-known (see e.g. \cite{PL07}) that
\begin{equation*}
F_{n}=\bigg\lfloor\frac{\rho^{n}}{\sqrt{5}}+\frac{1}{2}\bigg\rfloor,\quad n\in\N,
\end{equation*}
where $\lfloor x\rfloor$ is the largest integer not greater than $x$. Setting $\mu_{*}=\max\{\mu,1-\mu\}$, we then deduce from $\mu\in(1-\frac{1}{\rho},\frac{1}{\rho})$ that $\rho\mu_{*}<1$. It then follows that there's some $C_{*}>0$ such that
\begin{equation*}
\P\big\{\Om_{q}\big\}\leq\bigg(\frac{\rho^{L+2}}{\sqrt{5}}+\frac{1}{2}\bigg)\mu_{*}^{L}\leq C_{*}(\rho\mu_{*})^{L}
\end{equation*}
for all large $L\in2\N_{0}+1$. Setting $L=c(E-E_{0})^{-1/2}$, we find
\begin{equation*}
\P\big\{\om\in\Om\big|E_{L}(\om)\leq E\big\}\leq c^{d-1}(E-E_{0})^{-(d-1)/2}C_{*}(\rho\mu_{*})^{c(E-E_{0})^{-1/2}}.
\end{equation*}

\item[\rm(ii)] If the common distribution of the i.i.d random variables $\{\om_{i}\}_{i\in\Z^{d}}$ has a continuous density, then we can find some $\ep>0$ such that $\mu\in(1-\frac{1}{\rho},\frac{1}{\rho})$.
\end{itemize}
\end{rem}

\subsection{Bernoulli Models}\label{subsec-B-model}

We consider Bernoulli models, that is, the i.i.d random variables $\{\om_{i}\}_{i\in\Z^{d}}$ satisfy
\begin{equation*}
\begin{split}
&\P\{\om\in\Om|\om_{*}=a\}+\P\{\om\in\Om|\om_{*}=b\}=1,\\
&\P\{\om\in\Om|\om_{*}=a\}\P\{\om\in\Om|\om_{*}=b\}>0,
\end{split}
\end{equation*}
where $\om_{*}$ is the universal representation of $\{\om_{i}\}_{i\in\Z^{d}}$. Theorem \ref{thm-Lifshitz-tail-non} in this case is restated as

\begin{thm}\label{thm-Lifshitz-tail-non-optimal-B}
If the i.i.d random variables $\{\om_{i}\}_{i\in\Z^{d}}$ are Bernoulli distributed, then
\begin{equation*}
\limsup_{E\da E_{0}}\frac{\ln|\ln N(E)|}{\ln(E-E_{0})}\leq-\frac{1}{2}.
\end{equation*}
\end{thm}

The proof of the above theorem is based on the following

\begin{lem}\label{lem-estimate-key-7}
There exists some $C>0$ such that for all large $L\in2\N_{0}+1$ there holds
\begin{equation}\label{estimate-key-2}
E_{L}(\om)\geq E_{0}+\frac{C}{L^{2}}
\end{equation}
for all $\om\in\{a,b\}^{\Z^{d}\cap\La_{L}}$ satisfying the property: for any $q\in\Z^{d-1}\cap(-\frac{L}{2},\frac{L}{2})^{d-1}$ there exist four consecutive integers $r_{1},r_{2},r_{3},r_{4}\in\Z\cap(-\frac{L}{2},\frac{L}{2})$ with $r_{1}<r_{2}<r_{3}<r_{4}$ such that $\om_{(q,r_{1})}\neq\om_{(q,r_{2})}$ and $\om_{(q,r_{3})}\neq\om_{(q,r_{4})}$.
\end{lem}

We postpone the proof of Lemma \ref{lem-estimate-key-7} to the proof of Theorem \ref{thm-Lifshitz-tail-non-optimal-B}.

\begin{proof}[Proof of Theorem \ref{thm-Lifshitz-tail-non-optimal-B}]
It suffices to give a proper estimate for $\P\{\om\in\Om|E_{L}(\om)\leq E\}$ for large $L\in2\N_{0}+1$. Let $E>E_{0}$ and set $L=c(E-E_{0})^{-1/2}$ for some $c>0$ with $c^{2}<C$, where $C>0$ is the same as in \eqref{estimate-key-2}. We assume that $E$ is close to $E_{0}$ so that $L$ is large. If $\om\in\{a,b\}^{\Z^{d}\cap\La_{L}}$ satisfies Lemma \ref{lem-estimate-key-7}$\rm(ii)$, we deduce from \eqref{estimate-key-2} that $E_{L}(\om)>E$. Therefore,
\begin{equation*}
\begin{split}
&\P\big\{\om\in\Om\big|E_{L}(\om)\leq E\big\}\\
&\quad\quad\leq\P\Bigg\{\om\in\{a,b\}^{\Z^{d}\cap\La_{L}}\Bigg|
\begin{aligned}
&\exists q\in\Z^{d-1}\cap(-L/2,L/2)^{d-1}\,\,\text{s.t.}\,\,\text{any four consecutive integers}\\ &r_{1},r_{2},r_{3},r_{4}\in\Z\cap(-L/2,L/2)\,\,\text{with}\,\,r_{1}<r_{2}<r_{3}<r_{4}\\
&\text{satisfies}\,\,\om_{(q,r_{1})}=\om_{(q,r_{2})}\,\,\text{or}\,\,\om_{(q,r_{3})}=\om_{(q,r_{4})}
\end{aligned}\Bigg\}\\
&\quad\quad\leq\sum_{q\in\Z^{d-1}\cap(-\frac{L}{2},\frac{L}{2})^{d-1}}\P\Bigg\{\om\in\{a,b\}^{\Z^{d}\cap\SS_{qL}}\Bigg|
\begin{aligned}
&\text{any four consecutive integers}\,\,r_{1},r_{2},r_{3},r_{4}\in\\
&\Z\cap(-L/2,L/2)\,\,\text{with}\,\,r_{1}<r_{2}<r_{3}<r_{4}\\
&\text{satisfies}\,\,\om_{(q,r_{1})}=\om_{(q,r_{2})}\,\,\text{or}\,\,\om_{(q,r_{3})}=\om_{(q,r_{4})}
\end{aligned}\Bigg\}\\
&\quad\quad\triangleq\sum_{q\in\Z^{d-1}\cap(-\frac{L}{2},\frac{L}{2})^{d-1}}\P\big\{\Om_{q}\big\}\\
&\quad\quad=L^{d-1}\P\big\{\Om_{q_{0}}\big\}
\end{split}
\end{equation*}
for any $q_{0}\in\Z^{d-1}\cap(-\frac{L}{2},\frac{L}{2})^{d-1}$.

To find an upper bound for $\P\big\{\Om_{q_{0}}\big\}$, we use the argument as in the proof of Theorem \ref{thm-Lifshitz-tail-non-optimal}. Let $N$ be the largest integer such that $4N\leq\frac{L}{2}$. For $n=1,2,\dots,N$, we set
\begin{equation*}
I_{n}=\big\{4(n-1)+1,4(n-1)+2,4(n-1)+3,4(n-1)+4\big\}
\end{equation*}
and for $n=-1,-2,\dots,-N$, we set
\begin{equation*}
I_{n}=\big\{4(n+1)-4,4(n+1)-3,4(n+1)-2,4(n+1)-1\big\}.
\end{equation*}
That is, we decompose the sets $\{1,\dots,4N\}$ and $\{-4N,\dots,-1\}$ into disjoint sets such that each such set consists of four consecutive integers. Then, for any $n\in\{-N,\dots,N\}\bs\{0\}$, we can simply write
\begin{equation*}
I_{n}=\big\{r_{n1},r_{n2},r_{n3},r_{n4}\big\}
\end{equation*}
with $r_{n1}<r_{n2}<r_{n3}<r_{n4}$. Moreover, for any $m,n\in\{-N,\dots,N\}\bs\{0\}$ with $m\neq n$, $\om_{(q_{0},I_{m})}$ and $\om_{(q_{0},I_{n})}$ are independent. With these, we find
\begin{equation*}
\begin{split}
\P\big\{\Om_{q_{0}}\big\}&\leq\P\Bigg\{\om\in[a,b]^{\Z^{d}\cap\SS_{q_{0}L}}\Bigg|
    \begin{aligned}
    &\forall n\in\{-N,-N+1,\dots,N-1,N\}\bs\{0\},\\
    &\om_{(q_{0},r_{n1})}=\om_{(q_{0},r_{n2})}\,\,\text{or}\,\,\om_{(q_{0},r_{n3})}=\om_{(q_{0},r_{n4})}
    \end{aligned}\Bigg\}=\Big(P\big\{\Om_{q_{0}}(n_{0})\big\}\Big)^{2N}
\end{split}
\end{equation*}
for any $n_{0}\in\{-N,\dots,N\}\bs\{0\}$, where
\begin{equation*}
\P\big\{\Om_{q_{0}}(n)\big\}=\Big\{\om\in\{a,b\}^{(q_{0},I_{n})}\Big|\om_{(q_{0},r_{n1})}=\om_{(q_{0},r_{n2})}\,\,\text{or}\,\,\om_{(q_{0},r_{n3})}=\om_{(q_{0},r_{n4})}\Big\}
\end{equation*}
for $n\in\{-N,\dots,N\}\bs\{0\}$. It's not hard to check that
$\P\big\{\Om_{q_{0}}(n_{0})\big\}=1-4\mu_{a}^{2}\mu_{b}^{2}\in\big[\frac{3}{4},1\big)$, where $\mu_{a}=\P\{\om\in\Om|\om_{*}=a\}$ and $\mu_{b}=\P\{\om\in\Om|\om_{*}=b\}=1-a$. Therefore, by setting $\mu=1-4\mu_{a}^{2}\mu_{b}^{2}$, we arrive at
$\P\big\{\Om_{q_{0}}\big\}\leq\mu^{2N}\leq\mu^{-\frac{7}{4}}\mu^{\frac{L}{4}}$,
where we used $8N\geq L-7$. Consequently, recalling $L=c(E-E_{0})^{-1/2}$, we find
\begin{equation*}
\P\big\{\om\in\Om\big|E_{L}(\om)\leq E\big\}\leq L^{d-1}\mu^{-\frac{7}{4}}\mu^{\frac{L}{4}}=c^{d-1}(E-E_{0})^{-(d-1)/2}\mu^{-\frac{7}{4}}\mu^{\frac{c}{4}(E-E_{0})^{-1/2}}
\end{equation*}
for all $E>E_{0}$ with $E-E_{0}$ small, which leads to the result.
\end{proof}

The rest of this subsection is devoted to the proof of Lemma \ref{lem-estimate-key-7}. We start with

\begin{lem}\label{lem-estimate-key-5}
Let $r,r+1\in\Z\cap(-\frac{L}{2},\frac{L}{2})$ and set
$\SS=(-\frac{1}{2},\frac{1}{2})^{d-1}\times(r-\frac{1}{2},\frac{L}{2})$.
Then there exists some $C>0$ such that for all large $L\in2\N_{0}+1$ there holds
\begin{equation*}
\inf\si(H_{\om,\SS})\geq E_{0}+\frac{C}{\big(\frac{L+1}{2}-r\big)^{2}}
\end{equation*}
for all $\om\in\{a,b\}^{\Z^{d}\cap\SS}$ satisfying $\om_{(0,r)}\neq\om_{(0,r+1)}$.
\end{lem}
\begin{proof}
Let $\om\in\{a,b\}^{\Z^{d}\cap\SS}$ satisfy $\om_{(0,r)}\neq\om_{(0,r+1)}$. We claim that there are subsegments $\SS_{1},\dots,\SS_{K}$ satisfying following conditions:
\begin{itemize}
\item[\rm(i)] for each $k\in\{1,\dots,K\}$, there are $l_{k},m_{k}\in\Z\cap(-\frac{L}{2},\frac{L}{2})$ with $l_{k}+1\leq m_{k}$ such that $\SS_{k}=(-\frac{1}{2},\frac{1}{2})^{d-1}\times(l_{k}-\frac{1}{2},m_{k}+\frac{1}{2})$;

\item[\rm(ii)] $\SS_{1},\dots,\SS_{K}$ are pairwise disjoint and $\ol{\SS}=\cup_{k=1}^{K}\ol{\SS}_{k}$;

\item[\rm(iii)] for each $k\in\{1,\dots,K\}$, $\{\om_{(0,n)},n=l_{k},l_{k}+1,\dots,m_{k}\}$ satisfies the following conditions: $\om_{(0,l_{k})}\neq\om_{(0,l_{k}+1)}$ and $\om_{(0,n)}=a$ for all $n=l_{k}+2,\dots,m_{k}$ or $\om_{(0,n)}=b$ for all $n=l_{k}+2,\dots,m_{k}$.
\end{itemize}
This claim follows from a similar (in fact, much simpler) iteration argument as in the proof of Lemma \ref{lem-estimate-key-1}. The rest proof follows from Lemma \ref{app-lem-bracketing} and Theorem \ref{thm-lower-bound-gse-main} with $\rm(H5)$.
\end{proof}

We also need the counterpart.

\begin{lem}\label{lem-estimate-key-6}
Let $r-1,r\in\Z\cap(-\frac{L}{2},\frac{L}{2})$ and set
$\SS=(-\frac{1}{2},\frac{1}{2})^{d-1}\times(-\frac{L}{2},r+\frac{1}{2})$.
Then there exists some $C>0$ such that for all large $L\in2\N_{0}+1$ there holds
\begin{equation*}
\inf\si(H_{\om,\SS})\geq E_{0}+\frac{C}{\big(\frac{L+1}{2}+r\big)^{2}}
\end{equation*}
for all $\om\in\{a,b\}^{\Z^{d}\cap\SS}$ satisfying $\om_{(0,r)}\neq\om_{(0,r-1)}$.
\end{lem}

We now prove Lemma \ref{lem-estimate-key-7}.

\begin{proof}[Proof of Lemma \ref{lem-estimate-key-7}]
For large $L\in2\N_{0}+1$, let $\SS_{0L}=(-\frac{1}{2},\frac{1}{2})^{d-1}\times(-\frac{L}{2},\frac{L}{2})$ and set $\SS_{qL}=(q,0)+\SS_{0L}$ for $q\in\Z^{d-1}$. Lemma \ref{app-lem-bracketing} then implies
\begin{equation*}
E_{L}(\om)\geq\min_{q\in\Z^{d-1}\cap(-\frac{L}{2},\frac{L}{2})^{d-1}}\inf\si(H_{\om,\SS_{qL}}).
\end{equation*}

Now, let $\om\in\{a,b\}^{\Z^{d}\cap\La_{L}}$ be as in the statement of Lemma \ref{lem-estimate-key-7}. Fix any $q\in\Z^{d-1}\cap(-\frac{L}{2},\frac{L}{2})^{d-1}$ and consider the operator $H_{\om,\SS_{qL}}$. Let $r_{1},r_{2},r_{3},r_{4}\in\Z\cap(-\frac{L}{2},\frac{L}{2})$ be consecutive integers satisfying $r_{1}<r_{2}<r_{3}<r_{4}$, $\om_{(q,r_{1})}\neq\om_{(q,r_{2})}$ and $\om_{(q,r_{3})}\neq\om_{(q,r_{4})}$. Chopping $\SS_{qL}$ into
\begin{equation*}
\SS_{qL1}=\bigg(-\frac{1}{2},\frac{1}{2}\bigg)^{d-1}\times\bigg(-\frac{L}{2},r_{2}+\frac{1}{2}\bigg),\quad\SS_{qL2}=\bigg(-\frac{1}{2},\frac{1}{2}\bigg)^{d-1}\times\bigg(r_{3}-\frac{1}{2},\frac{L}{2}\bigg)
\end{equation*}
and using Lemma \ref{app-lem-bracketing}, we find
\begin{equation*}
\inf\si(H_{\om,\SS_{qL}})\geq\min\big\{\inf\si(H_{\om,\SS_{qL1}}),\inf\si(H_{\om,\SS_{qL2}})\big\}.
\end{equation*}
Then, we can apply Lemma \ref{lem-estimate-key-5} to $H_{\om,\SS_{qL2}}$ and Lemma \ref{lem-estimate-key-6} to $H_{\om,\SS_{qL1}}$ to conclude the result.
\end{proof}

\section{Further Discussions}\label{sec-dis}

We give a proof of the lower bound of Lifshitz tails in the cases $E_{\vp_{a}}(a)\leq E_{\vp_{a}}(b)$ and $E_{\vp_{b}}(a)\geq E_{\vp_{b}}(b)$ in Subsection \ref{subsec-lower-bound} and use Klopp and Nakamura's Bernoulli model constructed in \cite{KN09} to explain that Lifshitz tails may fail if $\rm(H5)$ fails in Subsection \ref{subsec-example}.

\subsection{Lifshitz Tails: Lower Bound}\label{subsec-lower-bound}

As mentioned in Section \ref{sec-intro}, the lower bound for Lifshitz tails has been proven in \cite[Theorem 0.2]{KN09}, whose proof is based on some techniques set up in \cite{Kl02} and \cite{KW02}. We here give a simple proof if $E_{\vp_{a}}(a)\leq E_{\vp_{a}}(b)$ or $E_{\vp_{b}}(a)\geq E_{\vp_{b}}(b)$ is true. The result is given by

\begin{thm}\label{thm-lower-bound}
Suppose $\rm (H1)$, $\rm(H2)$ and $\rm(H3)$. If either
\begin{itemize}
\item[\rm(i)] $E_{\vp_{a}}(a)\leq E_{\vp_{a}}(b)$ and $\P_{0}\big\{[a,a+\ep)\big\}\geq C\ep^{\ka}$ for $C>0$, $\ka>0$ and all $\ep>0$ small, or
\end{itemize}
\begin{itemize}
\item[\rm(ii)] $E_{\vp_{b}}(a)\geq E_{\vp_{b}}(b)$ and $\P_{0}\big\{(b-\ep,b]\big\}\geq C\ep^{\ka}$ for $C>0$, $\ka>0$ and all $\ep>0$ small,
\end{itemize}
is satisfied, then
\begin{equation*}
\liminf_{E\da E_{0}}\frac{\ln|\ln N(E)|}{\ln(E-E_{0})}\geq-\frac{d}{2}.
\end{equation*}
\end{thm}

We will only prove Theorem \ref{thm-lower-bound} in the case $E_{\vp_{a}}(a)\leq E_{\vp_{a}}(b)$. Therefore, all the Mezincescu boundary conditions below are defined via $\vp_{a}$. Our Method is based on the following observation.

\begin{lem}\label{lem-lower-bound-0}
Suppose $\rm (H1)$, $\rm(H2)$ and $\rm(H3)$. Let
\begin{equation*}
\HH_{a,\om}=-\De+\VV_{a}+\sum_{i\in\Z^{d}}\om_{i}u_{+}(\cdot-i),
\end{equation*}
where $\VV_{a}=V_{0}-a\sum_{i\in\Z^{d}}u_{-}(\cdot-i)$. If $E_{\vp_{a}}(a)\leq E_{\vp_{a}}(b)$, then $H_{\om}\leq\HH_{a,\om}$ with $E_{0}=\inf\si(\HH_{a,\om})$.
Moreover, $H_{\om,\La_{L}}^{X}\leq\HH_{a,\om,\La_{L}}^{X}$ for all $L\in\N$ and $X=D$ or $X=N$.
\end{lem}
\begin{proof}
Note
\begin{equation*}
\begin{split}
H_{\om}&=-\De+V_{0}+\sum_{i\in\Z^{d}}\om_{i}u_{+}(\cdot-i)-\sum_{i\in\Z^{d}}\om_{i}u_{-}(\cdot-i)\\
&\leq -\De+V_{0}+\sum_{i\in\Z^{d}}\om_{i}u_{+}(\cdot-i)-a\sum_{i\in\Z^{d}}u_{-}(\cdot-i).
\end{split}
\end{equation*}
This shows $H_{\om}\leq\HH_{a,\om}$. The inequality $H_{\om,\La_{L}}^{X}\leq\HH_{a,\om,\La_{L}}^{X}$ in the sense of quadratic forms follows directly from the structure of the potentials.

To show $E_{0}=\inf\si(\HH_{a,\om})$, we note that
\begin{equation*}
\inf\si(\HH_{a,\om})=\inf\si\Big(-\De+\VV_{a}+a\sum_{i\in\Z^{d}}u_{+}(\cdot-i)\Big)=\inf\si(H_{a}),
\end{equation*}
therefore, we only need to show $E_{0}=\inf\si(H_{a})$. But, by Lemma \ref{lem-Mezin} and Theorem \ref{thm-spectral-minimum}, we have $\inf\si(H_{a})=\inf\si(H_{a,\CC_{0}})=E_{\vp_{a}}(a)=E_{0}$. This completes the proof.
\end{proof}

As a simple consequence of the above lemma and \eqref{IDS-expectation-X}, we have

\begin{lem}\label{lem-lower-bound-1}
Suppose $\rm (H1)$, $\rm(H2)$ and $\rm(H3)$. If $E_{\vp_{a}}(a)\leq E_{\vp_{a}}(b)$, then $N(E)\geq\NN_{a}(E)$ for $E\in\R$, where $\NN_{a}$ is the IDS of $\HH_{a,\om}$.
\end{lem}

We remark that $\NN_{a}$ is well-defined since $\HH_{a,\om}$ is a standard continuum Anderson model (see e.g. \cite{KM07,Ve08}). Using the above lemma and the fact that $\inf\si(H_{\om})=\inf\si(\HH_{a,\om})$ by Lemma \ref{lem-lower-bound-0}, to prove the lower bound, it suffices to estimate a lower bound for $\NN_{a}$. 

\begin{proof}[Proof of Theorem \ref{thm-lower-bound}]
Note that the random operators $\HH_{a,\om}$ is a standard continuum Anderson models, and bottoms of their spectrum are nothing but $E_{0}$. Therefore, standard arguments (see e.g. \cite{KM07,KW05}) ensure that if $\P_{0}\big\{[a,a+\ep)\big\}\geq C\ep^{\ka}$ for some $C>0$, $\ka>0$ and all $\ep>0$ small, then there are constants $C_{1}>0$, $C_{2}>0$ and $C_{3}>0$ such that
\begin{equation*}
\NN_{a}(E)\geq C_{1}(E-E_{0})^{d/2}(C_{2}(E-E_{0})^{\ka})^{C_{3}(E-E_{0})^{-d/2}}
\end{equation*}
for all $E>E_{0}$ with $E-E_{0}$ small. Lemma \ref{lem-lower-bound-1} then leads to the result.
\end{proof}

\subsection{Klopp and Nakamura's Bernoulli Model}\label{subsec-example}

In Theorem \ref{thm-Lifshitz-tail-non-critical}, we require assumption $\rm(H5)$. Here, we employ Klopp and Nakamura's Bernoulli Model constructed in \cite{KN09} to argue that Lifshitz tails may fail if $\rm(H5)$ fails.

Let $\psi\in C^{2}(\CC_{0})$ be strictly positive, reflection symmetric and constant near $\pa\CC_{0}$. Denote this positive constant by $c_{0}$. Set $u=\frac{\De\psi}{\psi}$ and consider the random operator
\begin{equation}\label{Bernoulli-model}
H_{\om}=-\De+\sum_{i\in\Z^{d}}\om_{i}u(\cdot-i),
\end{equation}
where $\{\om_{i}\}_{i\in\Z^{d}}$ are i.i.d Bernoulli random variables with support $\{0,1\}$. Then, $a=0$ and $b=1$. It was proven in \cite{KN09} that this model fails to exhibit Lifshitz tails, but exhibits a van-Hove singularity. We show that it satisfies
\begin{equation*}
E_{\vp_{0}}(0)=E_{\vp_{0}}(1)\quad\text{and}\quad E_{\vp_{1}}(0)=E_{\vp_{1}}(1)
\end{equation*}
and fails to satisfy $\rm(H5)$.

For $E_{\vp_{0}}(0)=E_{\vp_{0}}(1)$, we note $H_{0}=-\De$,  so we can take $\vp_{0}\equiv1$, which yields the coincidence of Mezincescu boundary condition defined via $\vp_{0}$ and Neumann boundary condition. Since $\psi$ is constant near $\pa\CC_{0}$, it satisfies Neumann boundary condition on $\pa\CC_{0}$, so $\psi\in\DD(-\De_{\CC_{0}}^{N}+u)$. We conclude from $(-\De_{\CC_{0}}^{N}+u)\psi=0$ and the strict positivity of $\psi$ that $E_{\vp_{0}}(1)=\inf\si(-\De_{\CC_{0}}^{N}+u)=0$. Hence, $E_{\vp_{0}}(0)=0=E_{\vp_{0}}(1)$.

For $E_{\vp_{1}}(0)=E_{\vp_{1}}(1)$, since $\psi$ is constant near $\CC_{0}$, we conclude that $\psi$ is not only the ground state of $-\De_{\CC_{0}}^{N}+u$, but also the ground state of $-\De_{\CC_{0}}^{P}+u$, where the capital $P$ stands for periodic boundary condition. Thus, by periodic extension, the function $\vp_{1}=\sum_{i\in\Z^{d}}\psi(\cdot-i)$ is a continuously differentiable, strictly positive and $\Z^{d}$-periodic ground state of $H_{1}=-\De+\sum_{i\in\Z^{d}}u(\cdot-i)$ with $H_{1}\vp_{1}=0$. It follows that $E_{\vp_{1}}(1)=0$. Since $\psi$ is constant near $\pa\CC_{0}$, the Mezincescu boundary condition defined via $\vp_{1}$ and Neumann boundary condition coincide, which implies that $E_{\vp_{1}}(0)=\inf\si(-\De_{\CC_{0}}^{N})=0$. Therefore, $E_{\vp_{1}}(0)=0=E_{\vp_{1}}(1)$.

We next show that $H_{\om}$ fails to satisfy $\rm(H5)$. Since $E_{\vp_{0}}(0)=E_{\vp_{0}}(1)$ and $E_{\vp_{1}}(0)=E_{\vp_{1}}(1)$, and Mezincescu boundary conditions defined via $\vp_{0}$ and $\vp_{1}$ coincide with Neumann boundary condition as discussed above, the failure of $\rm(H5)$ follows from the explicit ground state of localized operators. More precise, let $\SS\subset\R^{d}$ be any nonempty open set such that $\SS=\text{int}(\cup_{i\in\Z^{d}\cap\SS}\overline{C}_{i})$, then $\inf\si(H_{\om,\SS})=0$ with ground state $\vp_{\om,\SS}$ satisfying
\begin{equation*}
\begin{split}
\vp_{\om,\SS}|_{\CC_{i}}=\left\{ \begin{aligned}
\vp(\cdot-i),&\quad\text{if}\quad\om_{i}=1,\\
c_{0},&\quad\text{if}\quad\om_{i}=0
\end{aligned}\right.
\end{split}
\end{equation*}
for any $i\in\Z^{d}\cap\SS$.

We point out that the reflection symmetry assumption on $\psi$, made in \cite{KN09}, is only for the reflection symmetry of $u$, so $E(0)=E(1)$, where $E(0)$ and $E(1)$ are the ground state energies of $-\De_{\CC_{0}}^{N}$ and $-\De_{\CC_{0}}^{N}+u$, respectively. The proof of $E_{\vp_{0}}(0)=E_{\vp_{0}}(1)$ and $E_{\vp_{1}}(0)=E_{\vp_{1}}(1)$ above is clearly independent of the reflection symmetry of $\psi$. Moreover, the proof in \cite{KN09} of the van-Hove singularity of the IDS of \eqref{Bernoulli-model} is mainly based on the explicit ground states (in terms of $\psi$) of $H_{\om}$ restricted to cuboids. But without the reflection symmetry of $u$, we can still use these explicit ground states. Therefore, after dropping the reflection symmetry assumption on $\psi$, the arguments in \cite{KN09} still apply and the IDS of \eqref{Bernoulli-model} exhibits the van-Hove singularity.

The above analysis is summarized as

\begin{thm}
Let $\psi\in C^{2}(\CC_{0})$ be strictly positive and constant near $\pa\CC_{0}$. Set $u=\frac{\De\psi}{\psi}$. Let $\vp_{0}$ and $\vp_{1}$ be the continuously differentiable, strictly positive and $\Z^{d}$-periodic ground states of $H_{0}=-\De$ and $H_{1}=-\De+\sum_{i\in\Z^{d}}u(\cdot-i)$, respectively. Considering the Bernoulli model \eqref{Bernoulli-model}, we have
\begin{itemize}
\item[$\bullet$] $E_{\vp_{0}}(0)=0=E_{\vp_{0}}(1)$ and $E_{\vp_{1}}(0)=0=E_{\vp_{1}}(1)$;
\item[$\bullet$] the IDS of $H_{\om}$ exhibits the van-Hove singularity near $0$.
\end{itemize}
\end{thm}


\begin{thebibliography}{99}


\bibitem{BLS08} J. Baker, M. Loss and G. Stolz, Minimizing the ground state energy of an electron in a randomly deformed lattice. \textit{Comm. Math. Phys.} 283 (2008), no. 2, 397-415.

\bibitem{BLS09} J. Baker, M. Loss and G. Stolz, Low energy properties of the random displacement model. \textit{J. Funct. Anal.} 256 (2009), no. 8, 2725–2740.

\bibitem{CL90} R. Carmona and J. Lacroix, \textit{Spectral theory of random Schr\"{o}dinger operators.} Probability and its Applications. Birkh\"{a}user Boston, Inc., Boston, MA, 1990.







\bibitem{Fo95} G. Folland, \textit{Introduction to partial differential equations.} Second edition. Princeton University Press, Princeton, NJ, 1995.

\bibitem{Gh07} F. Ghribi, Internal Lifshits tails for random magnetic Schr\"{o}dinger operators. \textit{J. Funct. Anal.} 248 (2007), no. 2, 387-427.

\bibitem{Gh08} F. Ghribi, Lifshits tails for random Schr\"{o}dinger operators with nonsign definite potentials. \textit{Ann. Henri Poincar\'{e}} 9 (2008), no. 3, 595-624.





\bibitem{Ki89} W. Kirsch, \textit{Random Schrödinger operators. A course.} Schr\"{o}dinger operators (S{\o}nderborg, 1988), 264–370, Lecture Notes in Phys., 345, Springer, Berlin, 1989.

\bibitem{Ki08} W. Kirsch, \textit{An invitation to random Schr\"{o}dinger operators. With an appendix by Fr\'{e}d\'{e}ric Klopp.} Panor. Synth\`{e}ses, 25, Random Schr\"{o}dinger operators, 1-119, Soc. Math. France, Paris, 2008.


\bibitem{KM07} W. Kirsch and B. Metzger, The integrated density of states for random Schr\"{o}dinger operators. Spectral theory and mathematical physics: a Festschrift in honor of Barry Simon's 60th birthday, 649-696, Proc. Sympos. Pure Math., 76, Part 2, Amer. Math. Soc., Providence, RI, 2007.


\bibitem{KW05} W. Kirsch and S. Warzel, Lifshits tails caused by anisotropic decay: the emergence of a quantum-classical regime. \textit{Math. Phys. Anal. Geom.} 8 (2005), no. 3, 257-285.










\bibitem{KS87} W. Kirsch and B. Simon, Comparison theorems for the gap of Schr\"{o}dinger operators. \textit{J. Funct. Anal.} 75 (1987), no. 2, 396-410.


\bibitem{Kl02} F. Klopp, Internal Lifshitz tails for Schr\"{o}dinger operators with random potentials. \textit{J. Math. Phys.} 43 (2002), no. 6, 2948-2958.


\bibitem{KLNS12-1} F. Klopp, M. Loss, S. Nakamura and G. Stolz, Localization for the random displacement model. \textit{Duke Math. J.} 161 (2012), no. 4, 587–621.

\bibitem{KLNS12-2} F. Klopp, M. Loss, S. Nakamura and G. Stolz, Understanding the random displacement model: from ground state properties to localization. Spectral analysis of quantum Hamiltonians, 183–219, \textit{Oper. Theory Adv. Appl.}, 224, Birkh\"{a}user/Springer Basel AG, Basel, 2012.

\bibitem{KN09} F. Klopp and S. Nakamura, Spectral extrema and Lifshitz tails for non-monotonous alloy type models. \textit{Comm. Math. Phys.} 287 (2009), no. 3, 1133-1143.

\bibitem{KN10} F. Klopp and S. Nakamura, Lifshitz tails for generalized alloy-type random Schr\"{o}dinger operators. \textit{Anal. PDE} 3 (2010), no. 4, 409-426.

\bibitem{FNNN03} F. Klopp, S. Nakamura, F. Nakano and Y. Nomura, Anderson localization for 2D discrete Schr\"{o}dinger operators with random magnetic fields. \textit{Ann. Henri Poincar\'{e}} 4 (2003), no. 4, 795-811.

\bibitem{KW02} F. Klopp and T. Wolff, Lifshitz tails for 2-dimensional random Schr\"{o}dinger operators. Dedicated to the memory of Tom Wolff. \textit{J. Anal. Math.} 88 (2002), 63-147.








\bibitem{Me87} G. A. Mezincescu, Lifschitz singularities for periodic operators plus random potentials. \textit{J. Statist. Phys.} 49 (1987), no. 5-6, 1181-1190.


\bibitem{Mi02} T. Mine, The uniqueness of the integrated density of states for the Schr\"{o}dinger operators for the Robin boundary conditions. \textit{Publ. Res. Inst. Math. Sci.} 38 (2002), no. 2, 355-385.




\bibitem{Na06} H. Najar, The spectrum minimum for random Schrödinger operators with indefinite sign potentials. \textit{J. Math. Phys.} 47 (2006), no. 1, 013515,13 pp.


\bibitem{Na00-1} S. Nakamura, Lifshitz tail for Schr\"{o}dinger operator with random magnetic field. \textit{Comm. Math. Phys.} 214 (2000), no. 3, 565-572.

\bibitem{Na00-2} S. Nakamura, Lifshitz tail for 2D discrete Schr\"{o}dinger operator with random magnetic field. \textit{Ann. Henri Poincar\'{e}} 1 (2000), no. 5, 823-835.



\bibitem{PL07} A. Posamentier and I. Lehmann, \textit{The (fabulous) Fibonacci numbers. With an afterword by Herbert A. Hauptman.} Prometheus Books, Amherst, NY, 2007.



\bibitem{RS78} M. Reed and B. Simon, \textit{Methods of modern mathematical physics. IV. Analysis of operators.} Academic Press, New York-London, 1978.



\bibitem{Si82} B. Simon, Schr\"{o}dinger semigroups. \textit{Bull. Amer. Math. Soc. (N.S.)} 7 (1982), no. 3, 447-526.



\bibitem{St01} P. Stollmann, \textit{Caught by disorder. Bound states in random media.} Progress in Mathematical Physics, 20. Birkh\"{a}user Boston, Inc., Boston, MA, 2001.



\bibitem{Ve08} I. Veseli\'{c}, \textit{Existence and regularity properties of the integrated density of states of random Schr\"{o}dinger operators.} Lecture Notes in Mathematics, 1917. Springer-Verlag, Berlin, 2008.

\end{thebibliography}
\end{document}